\documentclass[runningheads]{llncs}
\usepackage[T1]{fontenc}
\usepackage{amsmath,amssymb}

\usepackage{graphicx}

\usepackage{turnstile}
\usepackage{proof}
\usepackage{float}
\usepackage{hyperref}

\usepackage{color}

\spnewtheorem*{well-ordering_proof_CR}{Lemma~6}
{\bfseries}{\itshape}

\DeclareMathSymbol{\mhyph}{\mathalpha}{operators}{`-}

\newcommand{\deq}{\mathbin{\dot=}}
\newcommand{\dra}{\mathbin{\dot\rightarrow}}
\newcommand{\dfa}{\mathbin{\dot\forall}}

\newcommand*{\sub}{\mathrm{sub}}
\newcommand*{\subt}{\mathrm{subt}}
\newcommand*{\num}[1]{\underline{#1}}

\newcommand*{\pred}[1]{\mathrm{#1}}

\newcommand*{\gn}[1]{\ulcorner{#1}\urcorner}
  
\newcommand\pole{{\protect\mathpalette{\protect\polehelper}{\bot}}} \def\polehelper#1#2{\mathrel{\rlap{$#1#2$}\mkern3mu{#1#2}}}

\newcommand{\Lt}{\mathcal{L}_{\mathrm{T}}}
\newcommand{\Lr}{\mathcal{L}_{\mathrm{R}}}
\newcommand{\T}{\mathit{T}}
\newcommand{\F}{\mathit{F}}

\begin{document}
%
%\title{Axiomatising Classical Realisability}
\title{Tarskian Theories of Krivine's Classical Realisability}
\titlerunning{Tarskian Theories of Krivine's Classical Realisability}
%\author{First Author\inst{1}\orcidID{0000-1111-2222-3333} \and
%Second Author\inst{2,3}\orcidID{1111-2222-3333-4444} \and
%Third Author\inst{3}\orcidID{2222--3333-4444-5555}}
\author{Daichi Hayashi\inst{1} and
Graham E. Leigh\inst{2}}
\institute{
Independent Scholar
\email{daichinhayashi0611@gmail.com} \and
Department of Philosophy, Linguistics and Theory of Science, University of Gothenburg, G\"oteborg, Sweden 
\email{graham.leigh@gu.se}\\
}%\url{http://www.springer.com/gp/computer-science/lncs}}
\maketitle              % typeset the header of the contribution
\begin{abstract}
%The abstract should briefly summarize the contents of the paper in

This paper presents a formal theory of Krivine's classical realisability interpretation for first-order Peano arithmetic ($\mathsf{PA}$). 
To formulate the theory as an extension of $\mathsf{PA}$, we first modify Krivine's original definition to the form of number realisability, similar to Kleene's intuitionistic realisability for Heyting arithmetic.
By axiomatising our realisability with additional predicate symbols, 
we obtain a first-order theory $\mathsf{CR}$ which can formally realise every theorem of $\mathsf{PA}$.
%, which can be seen as a generalisation of the usual Tarskian typed truth theory $\mathsf{CT}$.
Although $\mathsf{CR}$ itself is conservative over $\mathsf{PA}$, adding a type of reflection principle that roughly states that ``realisability  implies  truth'' results in  $\mathsf{CR}$ being essentially equivalent to the Tarskian theory $\mathsf{CT}$ of typed compositional truth,
which is known to be proof-theoretically stronger than $\mathsf{PA}$.
%Thus, $\mathsf{CT}$ can be considered a formal theory of classical realisability.
We also prove that a weaker reflection principle which preserves the distinction between realisability and truth is sufficient for $\mathsf{CR}$ to achieve the same strength as $\mathsf{CT}$. 
Furthermore, we formulate transfinite iterations of $\mathsf{CR}$ and its variants, and then we determine their proof-theoretic strength. 

\end{abstract}
\section{Introduction}

Tarski \cite{tarski1983concept} presented a truth definition for a formal language by distinguishing the object language from the metalanguage. 
Although Tarski preferred a model-theoretic definition of truth, many researchers have also performed axiomatic studies to examine the logical and ontological principles underpinning a formal the conception of truth. 
%(e.g., \cite{feferman1991reflecting,friedman1987axiomatic,halbach2014axiomatic,leigh2010ordinal}). 
As a typical example of such attempts, the language $\mathcal{L}$ of classical first-order Peano arithmetic $(\mathsf{PA})$ is frequently selected as the object language, to which a fresh unary predicate $\T(x)$ for truth is added. 
The compositional truth theory $\mathsf{CT}$~(see Definition~\ref{defn:CT}) is a natural axiomatisation of Tarski's truth definition for $\mathcal{L}$, defined as an extension of $\mathsf{PA}$ by a finite list of axioms concerning $\T$.
Various hierarchical or self-referential definitions of truth and their axiomatisations have been proposed \cite{cantini1990theory,feferman1991reflecting,friedman1987axiomatic,halbach2014axiomatic,leigh2010ordinal,leitgeb2005truth}; however, most follow Tarski's paradigm, at least partially.% of the classical model theory.

As another semantic framework for classical theories, Krivine formulated classical realisability \cite{krivine2001typed,krivine2003dependent,krivine2009realizability}, which is a classical version of Kleene's intuitionistic realisability \cite{kleene1945interpretation}.
In the following, we briefly explain Krivine's classical realisability. 
There are two kinds of syntactic expressions, \emph{terms} $t$, which represent programs, and \emph{stacks} $\pi$, which represent evaluation contexts. 
A \emph{process} $t \star \pi$ is a pair of a term $t$ and a stack $\pi$.
In addition, we fix a set $\pole$ (\emph{pole}) of processes that is closed under several evaluation rules for processes. 
Then, for each sentence $A$, the set $\lvert A \rvert_{\pole}$ (\emph{realisers}) of terms and the set $\lVert A \rVert_{\pole}$ (\emph{refuters}) of stacks are defined inductively such that $t \star \pi \in \pole$ for any $t \in \lvert A \rvert_{\pole}$ and  $\pi \in \lVert A \rVert_{\pole}$. Here, a term $t$ \emph{universally  realises} $A$ if $t \in \lvert A \rvert_{\pole}$ for every pole $\pole$.
Then, Krivine proved that each theorem of second-order arithmetic is universally realisable.\footnote{Moreover, Krivine's classical realisability can be given to other strong theories, such as Zermelo--Fraenkel set theory \cite{krivine2001typed}.} 
In particular, we can take the empty set $\emptyset$ as a pole, and it is easy to show that if $\lvert A \rvert_{\emptyset}$ is not empty, then $A$ is true in the standard model $\mathbb{N}$ of arithmetic.
In summary, Krivine's (universal) realisability implies truth in $\mathbb{N}$; thus, Tarskian truth can be positioned in Krivine's general framework. 
%Axiomatic approaches to truth often have been based on their model-theoretic motivations. 
%The most famous example is Tarskian typed-theory of truth $\mathsf{CT}$ (see Definition~\ref{defn:CT}), whose formulation comes from the standard classical model theory.
%As another example, the self-referential theory $\mathsf{FS}$ (see Definition~\ref{FS}), given by Friedman and Sheard~\cite{friedman1987axiomatic}, has a revision model, a generalisation of Tarskian classical model.
%On the other hand, several proof-theoretic studies (e.g.,~\cite{cantini1989notes,cantini1990theory,friedman1987axiomatic}) have revealed that many theories of truth, except $\mathsf{CT}$ and $\mathsf{FS},$ have a $truth$-$as$-$provability$ interpretation, in which the notion of truth is reduced to computability, verifiability or justifiability in some sense.

The purpose of this paper is to axiomatise a formal theory $\mathsf{CR}$ for Krivine's classical realisability in a similar manner to $\mathsf{CT}$ for Tarski's truth definition. 
To clarify the relationship with $\mathsf{CT}$, we formulate $\mathsf{CR}$ over $\mathsf{PA}$.
As clarified in the above explanation of Krivine's realisability, we require additional vocabularies for a pole $\pole$ and the relations $t \in \lvert A \rvert_{\pole}$ and $\pi \in \lVert A \rVert_{\pole}$. With the help of G\"{o}del-numbering, they can be expressed by a unary predicate $x \in \pole$ and binary predicates $x \T y $ and $x \F y$, respectively.
Although Krivine's realisability uses $\lambda$-terms to express terms, stacks, and processes, we define realisers and refuters as natural numbers, similar to Kleene's number realisability.
Since our base theory is $\mathsf{PA}$, this modification can simplify the formulation of $\mathsf{CR}$ substantially (Section~\ref{subsec:classical_number_realisability}).

The remainder of this paper is organised as follows.
In Section~\ref{sec:classical_realisability} we define classical number realisability as a combination of Krivine's classical realisability and Kleene's intuitionistic number realisability.
By formalising our realisability, we obtain a first-order theory $\mathsf{CR}$ of compositional realisability (Definition~\ref{defn:CR}).
In Section~\ref{sec:formalised_realisation_of_PA}, we observe that our classical number realisability can realise every theorem of $\mathsf{PA}$, which is formalisable in $\mathsf{CR}$ (Theorem~\ref{formal_realisability_PA}).
Then, in Sections~\ref{sec:proof_theoretic_strength_of_CR}, \ref{append:well_ordering_CR}, 
we study the proof-theoretic strength of $\mathsf{CR}$.
First, $\mathsf{CR}$ is shown to be conservative over $\mathsf{PA}$ (Proposition~\ref{prop:conservativity_CR}). 
Then, we formulate a kind of reflection principle under which $\mathsf{CR}$ essentially amounts to $\mathsf{CT}$ (Proposition~\ref{prop:CT_and_CR_with_empty}). 
We also consider a weaker reflection principle that is sound with respect to any pole, and we prove that the principle can make $\mathsf{CR}$ as strong as $\mathsf{CT}$ (Theorem~\ref{thm:strength_CR+}).
In Section~\ref{sec:ramified_realisability}, we define, for each predicative ordinal $\gamma$, a system $\mathsf{RR}_{< \gamma}$, which is a transfinitely iterated version of $\mathsf{CR}$.
We also consider its extensions by the same reflection principles as for $\mathsf{CR}$, and then we determine their proof-theoretic strength.
Finally, potential future work is discussed in Section~\ref{sec:classical_real_future_work}.

This paper is an extended version of the conference paper~\cite{hayashi2024compositional}.
This paper adds a new section (Section~\ref{sec:ramified_realisability}), in which ramified theories and their proof-theoretic properties are studied.

%%%%%%%%%%%%%%%%%%%%%%%%%%%%%%%%%%%%%%%%%%%%%%%%%%%%%%%%%%%

\subsection{Conventions and Notations}
We introduce abbreviations for common formal concepts concerning coding and recursive functions. %, which are mainly based on \cite{leigh2013proof,leigh2016reflecting,leigh2010ordinal}.
We denote by \( \mathcal L \) the first-order language of $\mathsf{PA}$.
The logical symbols of $\mathcal{L}$ are $\to$, $\forall$ and \( = \).
The non-logical symbols are a constant symbol $\num 0$ and the function symbols for all primitive recursive functions. 
In particular, $\mathcal{L}$ has the successor function $x + 1,$ with which we can define \emph{numerals} $\num 0,\num 1,\num 2,\dotsc$. 
Thus, we identify natural numbers with the corresponding numeral.
We also employ the false equation $0=1$ as the propositional constant $\bot$ for contradiction, and then
the other logical symbols are defined in a standard manner, e.g., $\neg A = A \to \bot$.

The primitive recursive pairing function is denoted $\langle \cdot, \cdot \rangle$ with projection functions $(\cdot)_0$ and $(\cdot)_1$ satisfying $(\langle x,y \rangle)_0 = x$ and $(\langle x,y \rangle)_1 = y$. 
Sequences are treated as iterated pairing: $\langle x_0, x_1,\dots,x_k \rangle := \langle x_0, \langle x_1, \dots, x_k \rangle \rangle$.
Based on these constructors, each finite extension of \( \mathcal{L} \) is associated a fixed G\"{o}del coding (denoted \( \ulcorner e \urcorner \) where \( e \) is a finite string of symbols of the extended language) for which the basic syntactic constructions are primitive recursive.
In particular, \( \mathcal L \) contains a binary function symbol \( \sub \) representing the mapping \( \gn{A(x)} , n \mapsto \gn{A(\num n)} \) in the case that \( x \) is the only free variable of \( A(x) \), and binary function symbols \( \deq \), \( \dra \) and \( \dfa \) representing, respectively, the operations \( \gn s, \gn t \mapsto \gn{ s = t } \), \( \gn A, \gn B \mapsto \gn{ A \rightarrow B } \) and \( \gn x, \gn A \mapsto \gn{ \forall x A } \).
These operations will sometimes be omitted and we write \( \gn{s=t} \) and \( \gn{A\to B} \) for \( \deq( \gn s , \gn t ) \) and \( \dra(\gn A ,\gn B) \), etc.

We introduce a number of abbreviations for \( \mathcal L \)-expressions corresponding to common properties or operations on G\"odel codes.
The property of being the code of a variable is expressed by the formula \( \pred{Var}(x) \),
$\pred{ClTerm}(x)$ denotes the formula expressing that $x$ is a code of a \emph{closed} $\mathcal{L}$-term, and 
for a fixed extension \( \mathcal L' \) of \( \mathcal L \), $\pred{Sent}_{\mathcal{L}'}(x)$ expresses that $x$ is the code of a sentence.
The concepts above are primitive recursively definable meaning that the representing \( \mathcal L \)-formula is simply an equation between \( \mathcal L \)-terms.
Given a formula $A(x)$ with at most $x$ is free, $\gn {A(\dot{x})}$ abbreviates the term $\sub(\gn {A(x)} , x)$ expressing the code of the formula \( A(\num n) \) where \( n \) is the value of \( x \).
For sequences $\vec{x}$ of variables, $\gn {A(\dot{\vec{x}})} := \sub(\ulcorner A(\vec{x}) \urcorner, \vec{x})$ is defined similarly.

Quantification over codes is associated similar abbreviations:
\begin{itemize}
	\item \( \forall \ulcorner A \urcorner \in \pred{Sent}_{\mathcal{L}'}. \ B( \ulcorner A \urcorner ) \) abbreviates \( \forall x (\pred{Sent}_{\mathcal{L}'}(x) \to B( x )) \).
	\item  \( \forall \gn s .\ B( \gn{s}  )  \) abbreviates
		\(
			\forall x( \pred{ClTerm}(x) \to B( x ) )
		\).
%	\item \( \forall \gn { A \rightarrow B} \in \pred{Sent}_{\mathcal{L}'}.\  B( \gn{ A \rightarrow B } , \gn A , \gn B ) \) abbreviates
%		\[
%		\forall x \forall y(\pred{Sent}_{\mathcal{L}'}(x \dra y) \to B( x \dra y , x , y ) ).
%		\]
	\item \( \forall \gn { A_v } \in \pred{Sent}_{\mathcal{L}'}.\ B( v , \gn{ A } ) \) abbr.
		\(
		\forall x\forall v(\pred{Var}(v) \wedge \pred{Sent}_{\mathcal{L}'}( \dfa v x ) \to B( v, x ) ),
		\) namely quantification relative to codes of formulas with at most one distinguished variable free.
\end{itemize}

Partial recursive functions can be expressed in \( \mathcal L \) via the Kleene `T predicate' method.
The ternary relation $ x \cdot y \simeq z $ expresses that the result of evaluating the \( x \)-th partial recursive function on input \( y \) terminates with output \( z \).
Note that this relation has a $\Sigma^{0}_1$ definition in $\mathsf{PA}$ as a formula \( \exists w (\mathrm{T}_1(x,y,w) \land (w)_0 = z ) \) where \( \mathrm{T}_1 \) is primitive recursive.
It will be notationally convenient to use \( x \cdot y \) in place of a \emph{term} (with the obvious interpretation) though use of this abbreviation will be constrained to contexts in which potential for confusion is minimal.

With the above ternary relation we can express the property of two closed terms having equal value, via a $\Sigma^0_1$-formula $\pred{Eq}(y,z)$.
That is, \( \pred{Eq}(y,z) \) expresses that $y$ and $z$ are codes of closed $\mathcal{L}$-terms $s$ and $t$ respectively such that $s=t$ is a true equation. 

%%%%%%%%%%%%%%%%%%%%%%%%%%%%%%%%%%%%%%%%%%%%%%%%%%%%%%%%%%%%%%

\subsection{Classical Compositional Truth}

Tarskian truth for $\mathcal{L}$ is characterised inductively in the standard manner:
\begin{itemize}
	\item A closed equation $s=t$ is true iff $s$ and $t$ denote the same value in $\mathbb{N}$;
	\item $A \to B$ is true iff if $A$ is true, then $B$ is true;
	\item $\forall x A(x)$ is true iff $A(s)$ is true for all closed terms $s$.
\end{itemize}

Quantification over $\mathcal{L}$-sentences and the operations on syntax implicit in the above clauses can be expressed via a G\"{o}del-numbering. 
Thus, employing a unary (truth) predicate $\T$, a formal system $\mathsf{CT}$ can be defined corresponding, in a straightforward manner, to the Tarskian truth clauses.
%For simplicity, we sometimes suppress the parentheses of $\mathrm{T}(x)$ and just write $\mathrm{T}x$.

\begin{definition}[$\mathsf{CT}$]\label{defn:CT}
For a unary predicate $\T,$ let $\Lt = \mathcal{L} \cup \{ \T \}.$
The $\mathcal{L}_{\mathrm{T}}$-theory $\mathsf{CT}$ (compositional truth) consists of $\mathsf{PA}$ formulated for the language $\Lt$ plus the following three axioms.
\begin{description}
%\item[$(\mathsf{CT}_{Eq})$] $\forall \gn {s} \forall \gn t \forall \gn {A_v} \in \pred{Sent}_{\mathcal{L}}.\  \pred{Eq}(\ulcorner s \urcorner, \ulcorner t \urcorner) \to (\mathrm{T} \ulcorner A(s) \urcorner \to \mathrm{T} \ulcorner A(t) \urcorner)$.

\item[$(\mathrm{CT}_{=})$] $\forall \gn s , \gn t. \ \T ( s \deq t )  \leftrightarrow \pred{Eq}(s,t) $.

\item[$(\mathrm{CT}_{\to})$] $\forall \gn {A },\gn {B} \in \pred{Sent}_{\mathcal{L}}. \ \T\gn {A \to B} \leftrightarrow ( \T \gn A \to \T\gn B ) $.

\item[$(\mathrm{CT}_{\forall})$] $\forall \gn {A_v} \in \pred{Sent}_{\mathcal{L}}. \ \T\gn {\forall v A} \leftrightarrow \forall x \T \gn {A(\dot{x})} $.

\end{description}
\end{definition}

Two consequences of the above axioms are of particular relevance.
%The other axioms correspond to the above conditions for Tarski's truth.
%
The first is the observation that $\mathsf{CT}$ staisfies Tarski's Convention T \cite[pp.~187--188]{tarski1983concept} for formulas in \( \mathcal L \):

\begin{lemma}\label{fact:tarski_biconditional}
The Tarski-biconditional is derivable in \( \mathsf{CT} \) for every formula of \( \mathcal L \). That is, for each formula \( A(x_1, \dotsc, x_k) \) of \( \mathcal L \) in which only the distinguished variables occur free, we have
\[
\mathsf{CT} \vdash \forall x_1, \dotsc , x_k ( \T \ulcorner A(\dot{ x}_1, \dotsc, \dot{ x}_k ) \urcorner \leftrightarrow A(x_1 , \dotsc, x_k) ).
\]
\end{lemma}

Second is the `term regularity principle' stating that the truth value of each formula depends only on the value of terms and not their `structure'.
%For a formula \( A \), variable \( x \) and term \( s \).
Let \( \subt \) be a primitive recursive function function such that \( \subt \colon \gn{A(x)}, \gn x , \gn s \mapsto \gn{A(s)} \) for each formula \( A \), variable \( x \) and term \( s \).
\begin{lemma}\label{term-regular-CT}
	Provable in \( \mathsf{CT} \) is the term regularity principle:
	\[
	\forall \gn {s} \forall \gn t \forall \gn {A_v} \in \pred{Sent}_{\mathcal{L}}.\  \pred{Eq}(\gn s , \gn t ) \to (\T \gn {A( s)} \to \T \gn {A( t)})
	\]
	where the term \( \gn{A(s)} \) is shorthand for \( \subt(\gn{A}, v, \gn s) \), and \( \gn{A(t)} \) likewise.
\end{lemma}

%Since $\mathsf{CT}$ proves the consistency of $\mathsf{PA}$, $\mathsf{CT}$ is proof-theoretically stronger than $\mathsf{PA}$. Moreover, its exact strength is well known 

\section{Classical Realisability}\label{sec:classical_realisability}
We present a classical number realisability interpretation for $\mathsf{PA}$ and the corresponding axiomatic theory $\mathsf{CR}$.

\subsection{Classical Number Realisability}\label{subsec:classical_number_realisability}
We introduce a realisability interpretation for $\mathsf{PA}$ based on Krivine's classical realisability. 
For our setting we require several modifications from Krivine's original definition.
The first modification is about realisers. In Krivine's formulation, realisers are essentially lambda terms (cf. \cite{oliva2008krivine}). 
As we seek to formalise realisability over $\mathsf{PA}$, it is natural to assume that realisers are natural numbers,
%\footnote{Of course, by using appropriate G\"{o}del numbering, we could formalise Krivine's realisability itself in $\mathsf{PA}$. Alternatively, it would be possible to employ applicative theories \cite{troelstraa1988constructivism} as the base theory.}.
similar to Kleene's number realisability for the intuitionistic arithmetic \cite{kleene1945interpretation}.
%For this reason, our realisability will be a combination of Krivine's classical realisability and Kleene's number realisability for intuitionistic arithmetic 

Second, we must define the realisability and refutability conditions explicitly for equality. In the language of second-order arithmetic, the equality $a=b$ between $a$ and $b$ is definable by Leibniz equality 
$\forall X (a \in X \to b \in X)$,
which in Krivine's definition, determines the realisability and refutability conditions of the equation uniquely.
For a first-order language, equality is a primitive logical symbol and it is a matter of choice what constitutes a refutation of a closed equation.
In some sense, we take the most naive approach motivated by Kleene and Krivine's choices: a true equation is refuted by every element of the pole $\pole \subseteq \mathbb{N}$ and a false equation by every natural number.
Although a natural question, we do not delve into the possibility of other definitions.

The third modification involves the interpretation of the first-order universal quantifier $\forall x$. In Krivine's definition, it is interpreted \emph{uniformly}, i.e., a term $t$ realises a universal sentence $\forall x A$ when $t$ realises every instance $A(n).$ 
This definition is sufficient for the interpretation of second-order arithmetic because the set of natural numbers $\mathbb{N}$ is definable so the axiom of induction does not need to be realised explicitly.
In contrast, in case of the first-order arithmetic, realisers of  induction must be presented and the uniform interpretation is not ideal for this purpose.
As an alternative, we use Kleene's interpretation, where a realiser of a universal sentence $\forall x A$ is, in essence, the code of a recursive function that maps each natural number \( n \) to a realiser of $A(n)$.
The modifier `in essence' above is merely because it is the notion of `refuter' (not `realiser') which is primitive. The relation between the two in the case of quantifiers is qualified in Lemma~\ref{partial_compositionality}.

%Although we present classical realisability in the form of number realisability, the idea is essentially based on Krivine's original version \cite{krivine2001typed}.

In light of the above remarks, we introduce the following definitions.
%First, we fix a set $\pole \subseteq \mathbb{N}$ (henceforth referred to as the \emph{pole}), which is any set closed under converse computation:

\begin{definition}\label{defn:pole}
A \textbf{pole} $\pole$ is a subset of $\mathbb{N}$ such that it is conversely closed under computation:
for all \( e,m,n \in \mathbb N \), if $ e \cdot m \simeq n$ and $n \in \pole,$ then $\langle e, m \rangle \in \pole$. 
\end{definition}

Note that the empty set $\emptyset$ and natural numbers $\mathbb{N}$ trivially satisfy the above condition; thus, they are poles.

Given a pole $\pole$ and $\mathcal{L}$-sentence $A$, we define sets $\lVert A \rVert_{\pole} , \lvert A \rvert_{\pole} \subseteq \mathbb{N}$ of, respectively, \emph{refutations} (or a counter-proofs) and \emph{realisations} (or proofs) of $A$.
The sets are defined such that every pair $ \langle n, m \rangle \in \lvert A \rvert_\pole \times \lVert A \rVert_\pole $ of a realisation and refutation is an element of the pole $\pole$.
Thus, $\pole$ can be seen as the set of contradictions.

\begin{definition}\label{defn:classical_number_realisability}
Fix a pole $\pole.$
For each $\mathcal{L}$-sentence $A$, the sets $\lvert A \rvert_{\pole}, \lVert A \rVert_{\pole} \subseteq \mathbb{N}$ are defined as follows. 
The set $\lvert A \rvert_{\pole}$ is defined directly from $\lVert A \rVert_{\pole}$\textup{:}
\[
\lvert A \lvert_{\pole}\ = \{ n \in \mathbb{N} \mid \forall m \in \lVert A \rVert_{\pole}.\ \langle n, m \rangle \in \pole \}.
\]
The set $\lVert A \rVert_{\pole}$ is defined inductively:
\begin{itemize}
\item 
%\begin{equation}
$\lVert s=t \rVert_{\pole} =
\begin{cases}
\mathbb{N}, & \text{if }\mathbb{N} \not\models s=t,\\
\pole, & \text{otherwise.}
\end{cases}$\smallskip
%\end{equation}

\item $\lVert A \to B \rVert_{\pole} = \{ n \mid (n)_0 \in \lvert A \rvert_{\pole} \text{ and } (n)_1 \in \lVert B \rVert_{\pole} \}$.\medskip

\item $\lVert \forall x A \rVert_{\pole} = \{ n \mid (n)_1 \in \lVert A((n)_0) \rVert_{\pole} \}$.
%\item $||\forall x A|| := \underset{n}{\bigcup} ||A(\bar{n})|| $

\end{itemize}
\end{definition}
The motivation for the definitions of $\lVert A \rVert_{\pole}$ and $\lvert A \rvert_{\pole}$ should be clear.
A false equation is refuted by every number whereas a true equation is refuted only by `contradictions', i.e., elements of the pole.
A refutation of $A \to B$ is a pair \( \langle m , n \rangle \) for which $m$ realises $A$ and $n$ refutes $B$.  
A refutation of $\forall x A$ is a pair \( \langle m , n \rangle \) such that $n$ refutes $A(m)$.
Finally, a realiser of $A$ is a number \( n \) that contradicts all refutations of $A$, i.e., $\langle n, m \rangle \in \pole$ for every $m \in \lVert A \rVert_{\pole}$.
In particular, the closure condition on poles implies that every partial recursive function \( \lVert A \rVert_\pole \to \pole \) is a realiser of \( A \): if for every \( n \in \lVert A \rVert_\pole \), \( e \cdot n \) is defined and an element of \( \pole \) then, by definition, \( \langle e , n \rangle \in \pole \) for every \( n \in  \lVert A \rVert_\pole \).
It is also clear that every realiser of \( A \) induces a canonical partial recursive function mapping \(  \lVert A \rVert_\pole \) to \( \pole \).

%%%%%%%%%%%%%%%%%%%%%%%%%%%%%%%%%%%%%%%%%%%%%%%%%%%%%%%%%%%%

\subsection{Compositional Theory for Realisability}
\label{subsec:compositional_theory_for_realisability}

%Here, we formalise the classical number realisability presented in Section~\ref{subsec:classical_number_realisability}.
A corollary of Tarski's undefinability of truth, the language of $\mathsf{PA}$ is insufficient to express  classical realisability fully without expanding the non-logical vocabulary.
Let \( \Lr \) extend $\mathcal{L}$ by three new predicate symbols:
\begin{itemize}
\item a unary predicate $\pole$ for a pole (written \( x \in \pole \));
\item a binary predicate $ \F $ for refutation (written \( x \F y \));
\item a binary predicate $ \T $ for realisation  (written \( x \T y \)).\footnote{Although $\T$ is definable by $\F$ and $\pole,$ we introduce $\T$ as a primitive to simplify the notation.}
\end{itemize}
%Unless otherwise stated, we assume that $\mathsf{PA}$ is formulated for $\mathcal{L}^{+}$.

\begin{definition}[Compositional Realisability]\label{defn:CR}
The $\Lr$-theory $\mathsf{CR}$ extends $\mathsf{PA}$ formulated over $\Lr$ by the universal closures of the following axioms:
\begin{description}
\item[$(\mathrm{Ax}_{\pole})$] 
$x \cdot y \simeq z \to ( z \in \pole \to \langle x, y \rangle \in \pole) $

\item[$(\mathrm{Ax}_{\T})$] $\forall \gn A \in \pred{Sent}_{\mathcal{L}}. \ a \T \gn A \leftrightarrow \forall b  ( b \F \gn A \to \langle a, b \rangle \in \pole ) $

%\item[$(\mathsf{CR}_{Eq})$] $\forall \gn {s=t}, \gn {A_x} \in \pred{Sent}_{\mathcal{L}}. \ \pred{Eq}(\gn s ,\gn t) \to \forall x( x \mathrm{F} \gn {A(s)} \to x \mathrm{F} \gn {A(t)} ) $

\item[$(\mathrm{CR}_{=})$] $\forall \gn s, \gn t. \ a \F \gn {s = t}  \leftrightarrow \bigl( \pred{Eq}(\gn s, \gn t) \to a \in \pole \bigr) $

\item[$(\mathrm{CR}_{\to})$] $\forall \gn A ,\gn B \in \pred{Sent}_{\mathcal{L}}. \ a \F\gn {A \to B} \leftrightarrow \bigl((a)_0 \T \gn A \land (a)_1\F\gn B \bigr)$

\item[$(\mathrm{CR}_{\forall})$] $\forall \gn {A_x} \in \pred{Sent}_{\mathcal{L}}. \ a \F\gn {\forall x A} \leftrightarrow (a)_1 \F \gn {A (\dot{a})_0}    $
%$ \overrightarrow{s} \mathrm{F}\ulcorner \forall x A \urcorner \leftrightarrow \exists n \overrightarrow{s}\mathrm{F} \ulcorner A(\bar{n}) \urcorner$
\end{description}
%$\mathrm{P} \in \{ \mathrm{T}, \mathrm{F} \}$; 
%$A$ and $B$ are codes of $\mathcal{L}$-sentences.
\end{definition}

\begin{remark}
The universal closure of the axiom $(\mathsf{CR}_{=})$ is equivalent to the conjunction of the following:
\begin{description}
\item[$(\mathrm{CR}_{=1})$] $\forall \gn s \forall \gn t. \ \lnot \pred{Eq} ( \gn s , \gn t) \to \forall a .\ a \F\gn{ s = t } $;

\item[$(\mathrm{CR}_{=2})$] $\forall \gn s \forall \gn t. \ \pred{Eq}( \gn s , \gn t) \to \forall a ( a \F\gn { s = t } \leftrightarrow a \in \pole) $.
\end{description}
\end{remark}

A straightforward formal induction in \( \mathsf{CR} \) verifies the term regularity principle for refutations (cf.~Lemma~\ref{term-regular-CT}).

\begin{proposition}\label{term-regular}
	Refutations are provably invariant under term values:
	\[
	\forall \gn {s} \forall \gn t \forall \gn {A_v} \in \pred{Sent}_{\mathcal{L}}.\  \pred{Eq}(\gn s , \gn t ) \to \forall x( x \F \gn {A(s)} \to x \F \gn {A( t)}).
	\]
\end{proposition}

We provide a model of $\mathsf{CR}$ based on classical number realisability.
First, the interpretation of the vocabularies of $\mathcal{L}$ is naturally given by the standard model $\mathbb{N}$ of arithmetic.
Second, we fix any pole $\pole \subseteq \mathbb{N}$ for the interpretation of the predicate $x \in \pole$. The sets $\mathbb{T}_{\pole},\mathbb{F}_{\pole} \subseteq \mathbb{N} \times \mathbb{N}$ (for the interpretations of $x \T y$ and $x \F y$, respectively) are defined by the sets $\lvert A \rvert_{\pole}$ and $\lVert A \rVert_{\pole}$ in Definition~\ref{defn:classical_number_realisability}:

\begin{align}
&\mathbb{T}_{\pole} := \{ ( n,m ) \in \mathbb{N}^2 \mid m \ \text{is a code of an } \mathcal{L}\text{-sentence }A \text{ and }n \in \lvert A \rvert_{\pole} \},   \notag \\
&\mathbb{F}_{\pole} := \{ ( n,m ) \in \mathbb{N}^2 \mid m \ \text{is a code of an } \mathcal{L}\text{-sentence }A \text{ and } n \in \lVert A \rVert_{\pole} \}. \notag
\end{align}

Then, the following is clear.
\begin{proposition}\label{prop:soundness_CR}
	Let $\mathbb{N}$ be the standard model of $\mathcal{L}$
	and take any $\Lr$-sentence $A$. %let $A^*$ be its universal closure.
	If $\mathsf{CR} \vdash A,$ then the $\Lr$-model $\langle \mathbb{N}, \pole, \mathbb{T}_{\pole}, \mathbb{F}_{\pole} \rangle$ satisfies $A$.
\end{proposition}

\begin{proof}
The proof is by induction on the derivation of $A$.
If $A$ is an axiom of $\mathsf{PA}$, then the claim immediately holds.
Thus, it is enough to check each axiom of $\mathsf{CR}$.
For example, ($\mathsf{CR}_{\to}$) is satisfied for any pole $\pole$, any $\ulcorner A \urcorner, \ulcorner B \urcorner \in \pred{Sent}_{\mathcal{L}}$, and any $a \in \mathbb{N}$:
\begin{align}
\langle \mathbb{N}, \pole, \mathbb{T}_{\pole}, \mathbb{F}_{\pole} \rangle \models a \F \ulcorner A \to B \urcorner \ &\Leftrightarrow \ a \in \lVert A \to B \rVert_{\pole} \notag \\
&\Leftrightarrow \ (a)_0 \in \lvert A \rvert_{\pole} \ \& \ (a)_1 \in \lVert B \rVert_{\pole} \notag \\
&\Leftrightarrow \ \langle \mathbb{N}, \pole, \mathbb{T}_{\pole}, \mathbb{F}_{\pole} \rangle \models (a)_0 \T \ulcorner A \urcorner \land (a)_1 \F \ulcorner B \urcorner \notag 
\end{align}

Therefore, we have for any pole:
\[ 
\langle \mathbb{N}, \pole, \mathbb{T}_{\pole}, \mathbb{F}_{\pole} \rangle \models \forall \gn A ,\gn B \in \pred{Sent}_{\mathcal{L}}\,. \ \forall a \bigl( a \F \gn {A \to B} \leftrightarrow ((a)_0 \T \gn A \land (a)_1\F \gn B ) \bigr).
\] 

The other axioms are similarly satisfied. \qed
\end{proof}

%%%%%%%%%%%%%%%%%%%%%%%%%%%%%%%%%%%%%%%%%%%%%%

\section{Formalised Realisation of Peano Arithmetic}\label{sec:formalised_realisation_of_PA}

We demonstrate that the theory of classical number realisability realises every theorem of $\mathsf{PA}$.
In particular, we observe that this is formalisable in $\mathsf{CR}$. For that purpose, the following lemmas are useful.
%In the previous subsection, we showed that the reflection schema $\exists x (x\mathrm{T} \ulcorner A \urcorner) \to A$ (or equivalently the assumption $\pole = \emptyset$) gives $\mathsf{CR}$ the strength of $\mathsf{CT}.$ %In fact, we can easily verify that they are proof-theoretically equivalent.

%The remaining part of this paper is devoted to the proof of Lemma~\ref{well-ordering_proof_CR}.
%we show that in order for $\mathsf{CR}$ to get the strength of $\mathsf{CT},$ a weaker principle, $reflection \ rule,$ is sufficient. 
%For this purpose, we firstly show that every theorem of $\mathsf{PA}$ is realisable, provably in $\mathsf{CR}.$ Before that, the next lemma is useful. 

\begin{lemma}\label{partial_compositionality}
	There exists numbers $\mathsf{i}$, $\mathsf{u}$ and $\mathsf{s}$ such that
\begin{enumerate}
\item
\(
\mathsf{CR} \vdash \forall \gn A, \gn B \in \pred{Sent}_{\mathcal{L}}. \ a \T \gn{ A \to B } \wedge b \T \gn A \to ({\mathsf{i}} \cdot \langle a,b\rangle) T \ulcorner B \urcorner.
\)

\item %There exists a number $\mathsf{u}$ such that:
\(
\mathsf{CR} \vdash \forall \gn {A_x}  \in \pred{Sent}_{\mathcal{L}}. \ \forall x \bigl(( a\cdot x ) \T \gn {A(\dot{ x})} \bigr) \to (\mathsf{u}\cdot a)\T \gn {\forall x A}. 
\)

\item %There exists a number $\mathsf{s}$ such that:
\(
\mathsf{CR} \vdash \forall \gn {A_x} \in \pred{Sent}_{\mathcal{L}}. \ a \T \gn {\dfa x A} \to \forall y \bigl( ( \mathsf{s} \cdot \langle a,y \rangle) \T \gn {A(\dot{ y})} \bigr).
\)
\end{enumerate}
\end{lemma}

The meaning of these functions should be clear, i.e., $\mathsf{i}$ computes a realiser of $B$ from those for $A \to B$ and $A$, and
$\mathsf{u}$ expresses that if there exists a procedure that computes every instance $A(n),$ then $\forall x A$ is realised.
Conversely, $\mathsf{s}$ computes a realiser of $A(n)$ for each $n$.

\begin{proof}
%We deduce informally in $\mathsf{RCT}.$
\begin{enumerate}
\item Let $\mathsf{i}$ be such that $\mathsf{i} \cdot \langle a, b \rangle \simeq \lambda x. \langle a,b,x \rangle.$
To show that $\mathsf{i}$ is the required one, we take any $\mathcal{L}$-sentence $A \to B$ and assume $a \T \ulcorner A \to B \urcorner$ and $b \T \ulcorner A \urcorner$.
Then, we must show $(\mathsf{i} \cdot \langle a,b\rangle) \T \ulcorner B \urcorner$.
By the axiom $(\mathrm{Ax}_{\T}),$ taking any $c$ such that $c \F \ulcorner B \urcorner,$
we prove $\langle \mathsf{i} \cdot \langle a,b\rangle, c \rangle \in \pole.$
As $(\mathsf{i} \cdot \langle a,b\rangle) \cdot c \simeq \langle a,b,c \rangle,$ it is sufficient by the axiom $(\mathrm{Ax}_{\pole})$ to show that $\langle a,b,c \rangle$ is in $\pole.$
From the assumptions $b \T \ulcorner A \urcorner$ and $c \F \ulcorner B \urcorner$, as well as the axiom $(\mathrm{Ax}_{\T})$, we obtain $\langle b,c \rangle \F \ulcorner A \to B \urcorner.$
Thus, $(\mathrm{Ax}_{\T})$ yields that $\langle a,b,c \rangle = \langle a, \langle b,c \rangle \rangle \in \pole.$
 
\item Let $\mathsf{u}$ be such that $\mathsf{u} \cdot a \simeq \lambda x. \langle  a \cdot (x)_0, (x)_1 \rangle.$
To prove that this $\mathsf{u}$ is the required one, we take any $a$ and any $\mathcal{L}$-sentence $\forall x A.$
Then, under the assumption that $a$ is total and $\forall x .\, (a\cdot x) \mathrm{T} \ulcorner A(\dot{x}) \urcorner $ holds,
we must show $(\mathsf{u} \cdot a)\T \ulcorner \forall x A \urcorner$.
By the axiom $(\mathrm{Ax}_{\T})$, taking any $b$ such that $b \F \ulcorner \forall x A \urcorner,$
we prove $\langle \mathsf{u} \cdot a, b \rangle \in \pole.$
As $( \mathsf{u} \cdot a) \cdot b \simeq \langle  a \cdot (b)_0, (b)_1 \rangle,$ it is sufficient by the axiom $(\mathrm{Ax}_{\pole})$ to show the latter is in $\pole.$ From the assumption, we obtain the formula $ (a \cdot (b)_0) \T \ulcorner A((\dot{b})_0) \urcorner.$
In addition, by the axiom $(\mathrm{CR}_{\forall})$, we have $(b)_1 \F \ulcorner A((\dot{b})_0) \urcorner$.
Thus, the axiom $(\mathrm{Ax}_{\T})$ implies $\langle a \cdot (b)_0, (b)_1 \rangle \in \pole$.

\item Let $\mathsf{s}$ be such that $\mathsf{s} \cdot \langle a,b \rangle \simeq \lambda c. \langle a,b,c \rangle$,
and we show that this function is a required one. Thus, taking any $\ulcorner \forall x A(x) \urcorner \in \pred{Sent}_{\mathcal{L}}$ and any $a,b,c,$ we assume $a \T \ulcorner \forall x A \urcorner$
and $c \F \ulcorner A(\dot{b}) \urcorner.$
Then, by the axiom $(\mathrm{CR}_{\forall}),$ we obtain $\langle b,c \rangle \F \ulcorner \forall x A \urcorner.$
Thus, it follows that $\langle a, b, c \rangle \in \pole$ by the axiom $(\mathrm{Ax}_{\T})$.
Therefore, the axiom $(\mathrm{Ax}_{\pole})$ implies that $\langle \mathsf{s} \cdot \langle a,b \rangle, c \rangle \in \pole.$
Here, the $c$ is arbitrary; thus, we obtain $(\mathsf{s} \cdot \langle a,b \rangle) \T \ulcorner A(\dot{b}) \urcorner$ again by $(\mathrm{Ax}_{\T})$.
 \qed
\end{enumerate} 
\end{proof}

\begin{lemma}\label{lem:continuation_constant}
There are numbers \( \mathsf{k}_\pi \) and \( \mathsf{k}_\pole \) such that
\begin{enumerate}
\item %There is a number $\mathsf{k}_{\pi}$ such that: 
\(
	\mathsf{CR} \vdash \forall \gn A , \gn B \in \pred{Sent}_{\mathcal{L}}. \ a \F \gn A \to (\mathsf{k}_{\pi} \cdot a)\T \ulcorner A \to B \urcorner  .
\)
\item %There is a number $\mathsf{k}_{\pole}$ such that: 
\(
	\mathsf{CR} \vdash a \in \pole  \to \forall \ulcorner A \urcorner \in \pred{Sent}_{\mathcal{L}}. \ (\mathsf{k}_{\pole}\cdot a)\T \ulcorner A \urcorner.
\)
\end{enumerate}
\end{lemma}

\begin{proof}
\begin{enumerate}
\item Let $\mathsf{k}_{\pi} := \lambda a. \lambda b. \langle (b)_0, a \rangle.$ We prove that this $\mathsf{k}_{\pi}$ is the required number.
By taking any $a$ and any $\mathcal{L}$-sentence $A \to B,$ we assume $a \F \ulcorner A \urcorner$.
To demonstrate that $(\mathsf{k}_{\pi} \cdot a)\T \ulcorner A \to B \urcorner$, we take any $b$ such that $b \F \ulcorner A \to B \urcorner$, and then we must prove $\langle \mathsf{k}_{\pi} \cdot a, b \rangle \in \pole$. By the supposition $b \F \ulcorner A \to B \urcorner$ and the axiom $(\mathrm{CT}_{\to})$, it follows that $(b)_0 \T \ulcorner A \urcorner$.
Thus, we obtain $(\mathsf{k}_{\pi} \cdot a) \cdot b \simeq \langle (b)_0, a \rangle \in \pole$, which implies $\langle \mathsf{k}_{\pi} \cdot a, b \rangle \in \pole$ by the axiom $(\mathrm{Ax}_{\pole})$.
\item 
%We deduce informally in $\mathsf{RCT}.$
Assuming $a \in \pole$, we define a number $\mathsf{k}_{\pole} := \lambda a. \lambda b. a.$ % to be such that $f_{\pole} \cdot a \simeq \lambda b. a$ is true. 
To show $ (\mathsf{k}_{\pole} \cdot a)\T \ulcorner A \urcorner$, we take any $b$ such that $b\F \ulcorner A \urcorner.$ Then, $(\mathsf{k}_{\pole} \cdot a) \cdot b \simeq a \in \pole$; thus, the axiom 
$(\mathrm{Ax}_{\pole})$ implies $\langle \mathsf{k}_{\pole} \cdot a,b \rangle \in \pole.$
Therefore, $ (\mathsf{k}_{\pole} \cdot a)\T \ulcorner A \urcorner$ holds by the axiom $(\mathrm{Ax}_{\T})$. \qed
\end{enumerate}
\end{proof}

Note that the above  $\mathsf{k}_{\pi}$ is the CPS translation of \emph{call with current continuation} (cf. \cite{griffin1989formulae,krivine2003dependent}).
Using $\mathsf{k}_{\pi}$, we can define a realiser for Peirce's law.
In Krivine's formulation, Peirce's law is realised by the constant symbol $\mathsf{cc}$.
%\footnote{
%The symbol comes from the Felleisen--Griffin instruction 
%\emph{call with current continuation} (cf. \cite{griffin1989formulae}).}
Thus, our formulation is more similar to Oliva and Streicher's formulation of classical realisability \cite{oliva2008krivine} in that these constants are definable, i.e., they are not introduced as primitive symbols.

With the above preparations, we can now show the formalised realisability of $\mathsf{PA}$.
\begin{theorem}\label{formal_realisability_PA}
We assume that $\mathsf{PA}$ is formulated in the language $\mathcal{L}.$
For each $\mathcal{L}$-formula $A$, if $\mathsf{PA} \vdash A$, then there exists a closed term $s$ such that $\mathsf{CR} \vdash s \T \gn A $.
Moreover, this claim is formally expressible in $\mathsf{CR}$, i.e., we can find a number $\mathsf{k}_{\mathsf{PA}}$ such that: 
\begin{displaymath}
\mathsf{CR} \vdash \forall \ulcorner A \urcorner \in \pred{Sent}_{\mathcal{L}}. \ \pred{Bew}_{\mathsf{PA}}(x,\ulcorner A \urcorner) \to (\mathsf{k}_{\mathsf{PA}} \cdot x)\T \ulcorner A \urcorner,
\end{displaymath}
where $\pred{Bew}_{\mathsf{PA}}(x,y)$ is a canonical provability predicate for $\mathsf{PA}$, expressing that $x$ is a code of the proof of a sentence $y$. 
\end{theorem}

\begin{proof}
The proof is by induction on the length of the derivation of $A$ in $\mathsf{PA}$.
Here, we divide the cases by the last axiom or rule.

\textbf{Peirce's law} Assume that $A = ((B \to C) \to B) \to B$. According to Lemmas~\ref{partial_compositionality}~and~\ref{lem:continuation_constant}, we define a term $s$ as follows: 
\[s = \lambda b. \langle \mathsf{i} \cdot \langle (b)_0, \mathsf{k}_{\pi} \cdot (b)_1 \rangle, (b)_1 \rangle.\]
We show that $s \T \ulcorner ((B \to C) \to B) \to B \urcorner.$ Thus, taking any $b$ satisfying $b \F \ulcorner ((B \to C) \to B) \to B \urcorner,$
we prove $\langle s,b\rangle \in \pole.$ 
By the axiom $(\mathrm{Ax}_{\pole}),$ it is sufficient to show $\langle \mathsf{i} \cdot \langle (b)_0, \mathsf{k}_{\pi} \cdot (b)_1 \rangle, (b)_1 \rangle \in \pole.$
From the axiom $(\mathrm{CR}_{\to}),$ we obtain $(b)_0 \T \ulcorner (B \to C) \to B \urcorner$ and $(b)_1 \F \ulcorner B \urcorner$. Thus, by Lemma~\ref{lem:continuation_constant}, we obtain $(\mathsf{k}_{\pi} \cdot (b)_1) \T \ulcorner B \to C \urcorner,$
which implies that $ (\mathsf{i} \cdot \langle (b)_0, \mathsf{k}_{\pi} \cdot (b)_1 \rangle) \T \ulcorner B \urcorner$ by Lemma~\ref{partial_compositionality}.
Thus, by the axiom $(\mathrm{Ax}_{\T})$, we can derive the required formula $\langle \mathsf{i} \cdot \langle (b)_0, (\mathsf{k}_{\pi} \cdot (b)_1) \rangle, (b)_1 \rangle \in \pole$.

\textbf {Induction schema} Assume $A = B(0) \to ( \forall x (B(x) \to B(x+1)) \to \forall x B(x) )$.
We take any $b \F \gn {B(0) \to ( \forall x (B(x) \to B(x+1)) \to \forall x B(x) )} $.
Then, we obtain the following:
\begin{itemize} 
\item $(b)_0 \T \ulcorner B(0) \urcorner$;
\item $((b)_1)_0 \T \ulcorner \forall x (B(x) \to B(x+1)) \urcorner$; 
\item $((b)_1)_1 \F \ulcorner \forall x (B(x)) \urcorner$.
\end{itemize}
By the recursion theorem, choose a number $\mathsf{k}$ such that
\begin{enumerate}
	\item \( (\mathsf{k} \cdot b) \cdot 0 \simeq (b)_0  \)
	\item \( (\mathsf{k} \cdot b)\cdot (n+1) \simeq \mathsf{i} \cdot \langle \mathsf{s} \cdot \langle ((b)_1)_0, n \rangle, (\mathsf{k} \cdot b ) \cdot n \rangle \) for each $n$.
\end{enumerate}
%$\begin{cases} 
%(\mathsf{k} \cdot b) \cdot 0 :\simeq (b)_0 \\
%(\mathsf{k} \cdot b)\cdot (n+1) :\simeq \mathsf{i} \cdot \langle \mathsf{s} \cdot \langle ((b)_1)_0, n \rangle, (\mathsf{k} \cdot b ) \cdot n \rangle & \text{for each $n$}.
%\end{cases}$
Then, we obtain $\forall x .\, ((\mathsf{k} \cdot b) \cdot x) \T \ulcorner B(\dot{x})\urcorner $; thus,
Lemma~\ref{partial_compositionality} yields the formula $\langle \mathsf{u} \cdot (\mathsf{k} \cdot b), ((b)_1)_1 \rangle \in \pole.$
Therefore, for the term $s := \lambda b. \langle \mathsf{u} \cdot (\mathsf{k} \cdot b), ((b)_1)_1 \rangle,$
we have $\langle s, b \rangle \in \pole$ by the axiom $(\mathrm{Ax}_{\pole})$, and thus $s \T \ulcorner A \urcorner$ follows.

Note that the other cases are treated in a similar manner. \qed
\end{proof}

%%%%%%%%%%%%%%%%%%%%%%%%%%%%%%%%%%%%%%%%%%%%%

\section{Proof-Theoretic Strength of Compositional Realisability}\label{sec:proof_theoretic_strength_of_CR}
%\subsection{Proof-theoretic strength of $\mathsf{CR}$}
In Section~\ref{sec:formalised_realisation_of_PA}, we observed that $\mathsf{CR}$ is expressively strong enough to formalise the classical number realisability.
In the following, we turn to the proof-theoretic strength of $\mathsf{CR}$ and its relationship with $\mathsf{CT}$.
First, we show the conservativity of $\mathsf{CR}$ over $\mathsf{PA}$.
\begin{proposition}\label{prop:conservativity_CR}
$\mathsf{CR}$ is conservative over $\mathsf{PA}.$
\end{proposition}

\begin{proof}
We define a translation $\mathcal{T} \colon \Lr \to \mathcal{L}$ such that the vocabularies of $\mathcal{L}$ are unchanged, and then we show the following:
\begin{center} 
for $\Lr$-formula $A$, if $\mathsf{CR} \vdash A,$ then $\mathsf{PA} \vdash \mathcal{T}(A).$
\end{center}
If $A \in \mathcal{L},$ we have $\mathcal{T}(A) = A$; thus, the conservativity follows.

The translation $\mathcal{T}$ is defined as follows:
\begin{itemize}
\item $\mathcal{T}(s=t)= s=t$;
\item $\mathcal{T}(s \in \pole) = \mathcal{T}(s\F t) = \mathcal{T}(s\T t) = (0=0);$
\item $\mathcal{T}$ commutes with the logical symbols.
\end{itemize}
Roughly speaking, each pair is contradictory, and each sentence is realised and refuted by every number under this interpretation.
Therefore, we can easily see that the translation of each axiom of $\mathsf{CR}$ is derivable in $\mathsf{PA}.$ \qed
\end{proof}

\subsection{Compositional Realisability as Compositional Truth}

Although $\mathsf{CR}$ itself is proof-theoretically weak, here, we show that some assumption on the pole provides $\mathsf{CR}$ with the same strength as $\mathsf{CT}$. 

\begin{lemma}\label{lem:CR_with_empty}
Let $\pole = \emptyset$ denote the sentence $\neg \exists x (x \in \pole)$, and let $\mathsf{CR}^{\emptyset}$ be 
$\mathsf{CR}$ augmented with $\pole = \emptyset$. Then, $\mathsf{CR}^{\emptyset}$ can define the truth prediate of $\mathsf{CT}$
%\footnote{Therefore, $\mathsf{CT}$ is relatively interpretable in $\mathsf{CR}^{\emptyset}$. In particular, this is a relative truth definition in the sense of \cite{fujimoto2010relative}.} 
as, e.g., the predicate $0 \T x$. In other words, $\mathsf{CR}^{\emptyset}$ derives the following:
\begin{description}
%\item[$(\mathsf{CT}_{Eq})'$] $\forall \ulcorner s = t \urcorner, \ulcorner A(0) \urcorner \in \pred{Sent}_{\mathcal{L}}. \ Eq(\ulcorner s \urcorner, \ulcorner t \urcorner) \to (0\mathrm{T} \ulcorner A(s) \urcorner \to 0\mathrm{T} \ulcorner A(t) \urcorner)$;

\item[$(\mathrm{CT}_{=})'$] $\forall \gn s, \gn t. \ 0 \T \gn {s = t}  \leftrightarrow \pred{Eq}(\gn s, \gn t) $
%$\forall x \forall y ( \num 0\mathrm{T} \ulcorner \dot{x} = \dot{y} \urcorner \leftrightarrow x=y )$; 

\item[$(\mathrm{CT}_{\to})'$] $\forall \gn A , \gn B \in \pred{Sent}_{\mathcal{L}}. \ 0 \T\gn {A \to B} \leftrightarrow ( 0 \T \gn A \to 0 \T \gn B )$
%\\$\forall \gn A , \gn B \in \pred{Sent}_{\mathcal{L}}. \ 0\mathrm{T}\ulcorner A \to B \urcorner \leftrightarrow (0\mathrm{T} \ulcorner A \urcorner \to 0\mathrm{T}\ulcorner B \urcorner) $;

\item[$(\mathsf{CT}_{\forall})'$] $\forall \gn {A_x} \in \pred{Sent}_{\mathcal{L}}. \ 0\T\ulcorner \forall x A \urcorner \leftrightarrow \forall x \,0\T \ulcorner A(\dot{x}) \urcorner $.
\end{description}
Therefore, every $\mathcal{L}$-theorem of $\mathsf{CT}$ is derivable in $\mathsf{CR}^{\emptyset}$.
\end{lemma}

\begin{proof}
By $\pole = \emptyset$ and the axiom $(\mathrm{Ax}_{\T})$, we easily obtain the following: 
\begin{displaymath}
0\T\ulcorner A \urcorner %\leftrightarrow \exists x( x\mathrm{T}\ulcorner A \urcorner) 
\leftrightarrow
\forall b (\neg b \F \ulcorner A \urcorner). %\leftrightarrow \neg 0 \mathrm{F} \ulcorner A \urcorner.
\end{displaymath}
With this, we can derive the formulas $(\mathrm{CT}_{=})'$, $(\mathrm{CT}_{\to})'$, and $(\mathrm{CT}_{\forall})'$. 
\begin{description}
\item[$(\mathrm{CT}_{=})'$] In $\mathsf{CR}^{\emptyset}$, we deduce as follows.
	\begin{align*}
	0\T\gn{ s=t } &\Leftrightarrow \forall b (\neg b \F \gn{ s=t } ) &&\text{by $(\mathrm{Ax}_{\T})$ and $\pole = \emptyset$} \\
	&\Rightarrow \pred{Eq}(s, t) &&\text{by $(\mathrm{CR}_{=1}$)} \\
	&\Rightarrow \forall b (\neg b \F \gn{ s=t } )&& \text{by $(\mathrm{CR}_{=2})$ and $\pole = \emptyset$}
	\end{align*}
	
\item[$(\mathrm{CT}_{\to})'$] 
%Taking any code $\ulcorner A \to B \urcorner$ of an $\mathcal{L}$-sentence, we deduce as follows.
\begin{align*}
0\T\ulcorner A \to B \urcorner &\Leftrightarrow \neg \exists b (b\F \ulcorner A \to B \urcorner) &&\text{by $(\mathrm{Ax}_{\T})$ and $\pole = \emptyset$} \\
&\Leftrightarrow \neg \exists b ((b)_0\T \ulcorner A \urcorner \land (b)_1 \F \ulcorner B \urcorner) &&\text{by $(\mathrm{CR}_{\to})$} \\
&\Leftrightarrow \exists x (x\T\ulcorner A \urcorner) \to \neg \exists b (b\F \ulcorner B \urcorner) &&\text{by $\mathsf{PA}$} \\
&\Leftrightarrow 0\T\ulcorner A \urcorner \to 0\T\ulcorner B \urcorner
&&\text{by $(\mathrm{Ax}_{\T})$ and $\pole = \emptyset$}
\end{align*}

\item[$(\mathrm{CT}_{\forall})'$] 
\begin{align*}
0\T \gn {\forall x A} &\Leftrightarrow \forall b \, (\lnot b\F \gn {\forall x A}) &&\text{by $(\mathrm{Ax}_{\T})$ and $\pole = \emptyset$} \\
&\Leftrightarrow \forall b \, \bigl( \lnot  (b)_1\F \gn {A (\dot{b}_0 ) }  \bigr) &&\text{by $(\mathrm{CR}_{\forall} )$} \\
&\Leftrightarrow \forall x \forall y \, \lnot (y\F \ulcorner A(\dot{x}) \urcorner)  &&\text{by logic} \\
&\Leftrightarrow \forall x\, (0\T\ulcorner A(\dot{x}) \urcorner) &&\text{by $(\mathrm{Ax}_{\T})$ and $\pole = \emptyset$} 
\end{align*}
\end{description}
The other cases are similar. \qed
\end{proof}

Thus, the assumption $\pole = \emptyset$ reduces truth to realisability.
Next, we give another characterisation of this assumption using a kind of reflection principle stating that realisability is subsumed by truth.

\begin{lemma}\label{lem:reflection_empty_CR}
%\begin{enumerate}
Over $\mathsf{CR}$, the following are equivalent.
%\begin{enumerate}
%\item 
\begin{enumerate}
\item The reflection schema: \( \exists x (x\T \ulcorner A \urcorner) \to A  \)  for every $\mathcal{L}$-sentence $A$. \label{reflection_Lschema}
%\item 
\item The axiom: \( \pole = \emptyset \). \label{pole_is_empty}
\end{enumerate}
%\end{enumerate}
%\end{enumerate}
\end{lemma}

\begin{proof}

\begin{description}
\item[$(\ref{reflection_Lschema}) \Rightarrow (\ref{pole_is_empty})$:] Assume for a contradiction that $a \in \pole$ for some $a$. Then, $\mathsf{k}_{\pole} \cdot a$ verifies every $\mathcal{L}$-sentence by Lemma~\ref{lem:continuation_constant}.
Thus, the schema $(\ref{reflection_Lschema})$ implies every sentence, a contradiction. Therefore, the axiom $(\ref{pole_is_empty}):\pole = \emptyset$ follows.
\item[$(\ref{pole_is_empty}) \Rightarrow (\ref{reflection_Lschema})$:] 
%If $\pole = \emptyset,$ then by the axiom $(\mathsf{Ax}_{\mathrm{T}})$ we have: 
%\begin{displaymath}
%0\mathrm{T}\ulcorner A \urcorner \leftrightarrow \exists x( x\mathrm{T}\ulcorner A \urcorner) \leftrightarrow
%\forall b (\neg b \mathrm{F} \ulcorner A \urcorner) \leftrightarrow \neg 0 \mathrm{F} \ulcorner A \urcorner.
%\end{displaymath}
%From this and the axioms of $\mathsf{RCT},$ we can derive the axioms of $\mathsf{CT}:$
%\begin{description}
%\item[$(\mathsf{CT}_{P})'$] $0\mathrm{T} \ulcorner P(\dot{\overrightarrow{x}}) \urcorner \leftrightarrow P(\overrightarrow{x})$ for an $\mathcal{L}$-atomic predicate $P;$ 
%\item[$(\mathsf{CT}_{\to})'$] $\forall \ulcorner A \to B \urcorner \in \pred{Sent}_{\mathcal{L}} \{ 0\mathrm{T}\ulcorner A \to B \urcorner \leftrightarrow (0\mathrm{T} \ulcorner A \urcorner \to 0\mathrm{T}\ulcorner B \urcorner) \};$
%\item[$(\mathsf{CT}_{\forall})'$] $\forall \ulcorner \forall x A \urcorner \in \pred{Sent}_{\mathcal{L}} \{ 0\mathrm{T}\ulcorner \forall x A \urcorner \leftrightarrow \forall x 0\mathrm{T} \ulcorner A(\dot{x}) \urcorner \}.$
%\end{description}
As shown in Lemma~\ref{lem:CR_with_empty}, $\mathsf{CR}$ with $\pole = \emptyset $ can define the truth predicate of $\mathsf{CT}$ as $0\T x.$
Thus, similar to Lemma~\ref{fact:tarski_biconditional}, we obtain $0\T \ulcorner A \urcorner \to A$ for every $\mathcal{L}$-sentence $A$.
Therefore, we also have schema $(\ref{reflection_Lschema}).$ \qed
\end{description}
\end{proof}

\begin{proposition}\label{prop:CT_and_CR_with_empty}
%$\mathsf{RCT}$ with $\pole = \emptyset$ has a relative interpretation in $\mathsf{CT}.$
The theories $\mathsf{CR}^{\emptyset}$
%, $\mathsf{CR}$ with the reflection schema, 
and $\mathsf{CT}$ have exactly the same $\mathcal{L}$-consequences.
\end{proposition}

\begin{proof}
In Lemma~\ref{lem:CR_with_empty}, we observed that $\mathsf{CT}$ is interpretable in $\mathsf{CR}^{\emptyset}$.
%; thus, it is interpretable in $\mathsf{CR}$ with the reflection schema (Lemma~\ref{lem:reflection_empty_CR}).

For the converse direction, we note that the model construction of $\mathsf{CR}$ in Proposition~\ref{prop:soundness_CR} is also applicable to $\mathsf{CR}^{\emptyset}$ and is formalisable in the theory $\mathsf{ACA},$ the second-order system for arithmetical comprehension, which has the same $\mathcal{L}$-consequences as $\mathsf{CT}$ (for the proof, see, e.g., \cite{halbach2014axiomatic}). 
Alternatively, Lemma~\ref{lem:upper-bound_RR-empty} in Section~\ref{sec:ramified_realisability} provides a direct relative interpretation of $\mathsf{CR}^{\emptyset}$ in $\mathsf{CT}$.
\qed
%we use the translation $\mathcal{T}$ of Proposition~\ref{conservativity_RCT}, with the following change:
%\begin{itemize}
%\item $\mathcal{T}(s \in \pole) = 0\neq0;$
%\item $\mathcal{T}(s\mathrm{T}t) = \mathrm{T}t;$
%\item $\mathcal{T}(s\mathrm{F}t) = \mathrm{F}t.$
%\end{itemize}
%Then, it is easy to see that this translation gives a relative interpretation of $\mathsf{RCT}$ with $\pole = \emptyset$ into $\mathsf{CT}$.
%Since the translation $\mathcal{T}$ preserves $\mathcal{L}$-vocabularies, every $\mathcal{L}$-theorem of $\mathsf{RCT}$ with $\pole = \emptyset$ is derivable in $\mathsf{CT}$.
\end{proof}

\subsection{Compositional Realisability with the Reflection Rule}
Although $\mathsf{CR}^{\emptyset}$ (or equivalently $\mathsf{CR}$ with the reflection schema) and $\mathsf{CT}$ have the same proof-theoretic strength, 
$\mathsf{CR}^{\emptyset}$ is satisfied only when the pole $\pole$ is empty.
%realisability then becomes nothing other than truth.
%it cannot distinguish each realiser:
%\[
%\mathsf{CR}^{\emptyset} \vdash \forall a \forall b (a \mathrm{T} \ulcorner A \urcorner \leftrightarrow b \mathrm{T} \ulcorner A \urcorner).
%\]
Thus, our next goal is to find a principle that is compatible with any choice of the pole.
%In order for the classical realisability to be meaningful, 
Here, our suggestion is to weaken the reflection schema to the rule form.
\begin{definition}[$\mathsf{CR}^{+}$]\label{defn:reflection_rule}
The $\mathcal{L}_{R}$-theory $\mathsf{CR}^{+}$ is the extension of $\mathsf{CR}$ with the reflection rule:
\begin{displaymath}
\infer[]{A}{s \T \ulcorner A \urcorner}
\end{displaymath}
for every closed term $s$ and $\mathcal{L}$-sentence $A.$
\end{definition}

\begin{proposition}\label{prop:soundness_CR+}
Let $A$ be any $\Lr$-sentence and assume that $\mathsf{CR}^+ \vdash A$.
Then, for any pole $\pole$, we have $\langle \mathbb{N}, \pole, \mathbb{T}_{\pole}, \mathbb{F}_{\pole} \rangle \models A$.
\end{proposition}

\begin{proof}
The proof is by induction on the derivation of $A$.
Since the other cases are already contained in the proof of Proposition~\ref{prop:soundness_CR}, it is sufficient to consider the case of the reflection rule. Therefore, we assume that an $\mathcal{L}$-sentence $A$ is derived by the reflection rule from $s \T \ulcorner A \urcorner$ for some closed term $s$.  By the induction hypothesis, it follows that $\langle \mathbb{N}, \pole, \mathbb{T}_{\pole}, \mathbb{F}_{\pole} \rangle \models s \T \ulcorner A \urcorner$ for any pole $\pole$. 
Thus, particularly for the empty pole $\pole = \emptyset$, we obtain $\langle \mathbb{N}, \emptyset, \mathbb{T}_{\emptyset}, \mathbb{F}_{\emptyset} \rangle \models A$ by Proposition~\ref{prop:soundness_CR} and Lemma~\ref{lem:reflection_empty_CR}.
As $A$ is an $\mathcal{L}$-sentence,
we also have $\langle \mathbb{N}, \pole, \mathbb{T}_{\pole}, \mathbb{F}_{\pole} \rangle \models A$ for any pole $\pole$.  \qed
\end{proof}

In the next section, we prove that $\mathsf{CR}^+$ has the same proof-theoretic strength as $\mathsf{CT}$ and $\mathsf{CR}^{\emptyset}$. 
The upper bound of $\mathsf{CR}^+$ is obvious:
Lemma~\ref{lem:reflection_empty_CR} and Proposition~\ref{prop:soundness_CR+} establish that
$\mathsf{CR}^+$ is a \emph{proper} subtheory of $\mathsf{CR}^{\emptyset}$. 
%Nevertheless, $\mathsf{CR}^+$ is shown to be $\mathcal{L}$-conservative over $\mathsf{CT}^{\emptyset}$.
%The lower-bound proof of $\mathsf{CR}^+$ is given in Appendix~\ref{append:well_ordering_CR}. 
The lower-bound argument is more difficult
because $\mathsf{CR}^+$ is not expressively rich enough to interpret $\mathsf{CT}$.
Instead we can proceed directly through a well-ordering proof for $\mathsf{CR}^+$ by showing that the principle of transfinite induction for \( \mathcal L \)-formulas is provable for each ordinal below \( \varepsilon_{\varepsilon_0} \).
The argument is, essentially, just the extraction of the computational content of the  standard well-ordering proof for $\mathsf{CT}$ (for the detailed proof, see Section~\ref{append:well_ordering_CR}).
 
As a result, we can determine the proof-theoretic strength of $\mathsf{CR}^+$.

\begin{theorem}\label{thm:strength_CR+}
$\mathsf{CR}^{+}$ has exactly the same $\mathcal{L}$-theorems as $\mathsf{CT}$ and $\mathsf{CR}^{\emptyset}$.
\end{theorem}

%\begin{proof}
%The lower bound of $\mathsf{CR}^{+}$ is obvious by Lemma~\ref{well-ordering_proof_CR} and the reflection rule.
%For the upper bound, we remark that $\mathsf{CR}^{+}$ is a subsystem of $\mathsf{CR}^{\emptyset}$,
%which, by Proposition~\ref{prop:CT_and_CR_with_empty}, is proof-theoretically equivalent to $\mathsf{CT}$ and $\mathsf{PA} + \mathsf{TI}({<} \varepsilon_{\varepsilon_0}).$ \qed
%\end{proof}

%%%%%%%%%%%%%%%%%%%%%%%%%%%%%%%%%%%%%%%%%%%%%%%%

\section{Well-ordering Proof in Compositional Realisability}\label{append:well_ordering_CR}

Here, we determine the proof-theoretic strength of $\mathsf{CR}^+$.
However, in contrast to $\mathsf{CR}^{\emptyset}$, $\mathsf{CR}^+$ is not sufficiently expressively strong to relatively interpret $\mathsf{CT}$.
Thus, we provide a well-ordering proof of $\mathsf{CR}^+$, from which we can conclude that $\mathsf{CR}^+$ derives the same $\mathcal{L}$-consequences as both $\mathsf{CT}$ and $\mathsf{CR}^{\emptyset}$. 

For this purpose, we require an ordinal notation system $\pred{OT}$ for predicative ordinal numbers.
We use several facts about $\pred{OT}$ (for the proof, see, e.g.,~\cite{pohlers2008proof}).
A formula $x \in \pred{OT}$ is defined as meaning that $x$ is a representation of an ordinal number in $\pred{OT}.$
Let $\alpha, \beta,$ and $\gamma$ range over the ordinal numbers in $\pred{OT}$.
Thus, $\forall \alpha A(\alpha)$ abbreviates $\forall x (x \in \pred{OT} \to A(x)).$
By the standard method, we can define relations and operations on $\pred{OT}.$ Let $<$ be the less-than relation,
$0$ is zero as the ordinal number, $\alpha \in \pred{Suc}$ means that $\alpha$ is a successor ordinal, $\alpha \in \pred{Lim}$
says that $\alpha$ is a limit ordinal,
$+$ is the ordinal addition, $\varphi xy$ is the Veblen function, and the binary primitive recursive function $[\alpha]_x$ returns the $x$-th element of the fundamental sequence for $\alpha$. 
For these symbols, the following notations and facts are used:
\begin{itemize}
\item Let $\forall \alpha < \beta (A)$ mean $\forall \alpha (\alpha < \beta \to A)$.

\item Let $\omega^{\alpha}:= \varphi0\alpha$ and $1:= \omega^{0}.$

\item For \( \alpha < \omega^\alpha \) set
$[\omega^{\alpha}]_n =
\begin{cases}
\omega^{\alpha-1}\times n & \text{if $\alpha \in Suc,$} \\
\omega^{[\alpha]_n} & \text{if $\alpha \in Lim$}.
\end{cases}$

\item Let $\omega_n(\alpha):=
\begin{cases}
\alpha & \text{if $n=0,$} \\
\omega^{\omega_{n-1}(\alpha)} & \text{if $n>0.$} 
\end{cases}$

\item Let $\varepsilon_{\alpha}:=\varphi1\alpha.$
%which is the limit of the sequence $\{ 1, \varepsilon_{1}, \varepsilon_{\varepsilon_{1}},
%\varepsilon_{\varepsilon_{\varepsilon_{1}}}, \cdots \}$.

\item 
$[\varepsilon_{\alpha}]_n = %\omega_n(\varepsilon_{\alpha-1} + 1),$
\begin{cases}
\omega_n(1) & \text{if $\alpha = 0,$} \\
\omega_n(\varepsilon_{\alpha-1} + 1) & \text{if $\alpha \in Suc,$} \\
\varepsilon_{[\alpha]_n} & \text{if $\alpha \in Lim$}.
\end{cases}$

\end{itemize}

The $\mathrm{T}$-free consequences of $\mathsf{CT}$ can be expressed using some transfinite induction schema. For an $\mathcal{L}$-formula $A(x)$ and an ordinal $\alpha,$ 
we define the formula $\mathrm{TI}(A, \alpha) = \mathrm{Prog}\lambda x A(x) \to A(\alpha)$,
where $\mathrm{Prog}\lambda x A(x) = \forall \alpha(\forall \beta < \alpha (A(\beta)) \to A(\alpha)).$
Then, we define the schema:
%\[
$\mathrm{TI}(\alpha) = \{ \mathrm{TI}(A, \beta) \ | \ A\text{ is an $\mathcal{L}$-formula} \}$.
%\]
In addition, let
$\mathrm{TI}(<\alpha) := \bigcup_{\beta < \alpha} \mathrm{TI}(\beta).
$
Finally, let the $\mathcal{L}$-theory $\mathsf{PA} + \mathsf{TI}({<} \alpha)$ be the extension of $\mathsf{PA}$ with the schema $\mathrm{TI}({<} \alpha)$.

The following is well-known (see, e.g., \cite[Theorem~8.35]{halbach2014axiomatic}):
\begin{theorem}\label{fact:strength_CT}
The theory $\mathsf{CT}$ derives the same $\mathcal{L}$-formulae as $\mathsf{PA} + \mathsf{TI}({<} \varepsilon_{\varepsilon_{0}})$.
\end{theorem}

According to this fact, it is sufficient to show that the schema $\mathrm{TI}({<} \varepsilon_{\varepsilon_{0}})$ is derivable in $\mathsf{CR^+}$. 
To this end, in Lemma~\ref{well-ordering_proof_CR}, we prove that each instance of $\mathrm{TI}({<} \varepsilon_{\varepsilon_{0}})$ is realisable in $\mathsf{CR}$.
Then, $\mathsf{CR}^+$ can derive $\mathrm{TI}({<} \varepsilon_{\varepsilon_{0}})$ itself according to the reflection rule. 

The proof is essentially based on the standard well-ordering proof of $\mathsf{CT}$ (cf.~\cite[Lemma~3.11]{leigh2016reflecting}).
%However, we now need to give a realiser for the transfinite induction as the result of computation, 
%and hence the more careful analysis of the well-ordering proof of $\mathsf{CT}$ is necessary.
So, we briefly sketch the outline of the proof in $\mathsf{CT}$.
Let $I_0(\alpha) = \forall \ulcorner A \urcorner \in \pred{Sent}_{\mathcal{L}}. \ \T \ulcorner \mathrm{TI}(A, \dot{\alpha}) \urcorner$, which expresses the schema  $\mathrm{TI}(\alpha)$ as a single statement.
In $\mathsf{PA}$, we can derive the schemata $\mathrm{TI}(0)$, $\mathrm{TI}(\alpha+1)$ from $\mathrm{TI}(\alpha)$,
and $\mathrm{TI}(\omega^{\alpha})$ from $\mathrm{TI}(\alpha)$, respectively (e.g., see \cite[Section~7.4]{pohlers2008proof}).
By formalising these results, $\mathsf{CT}$ can derive $I_0(\varepsilon_0)$.
Furthermore, by generalising this argument in $\mathsf{CT}$,
we can derive $I_0(\varepsilon_{\alpha}) \to I_0(\varepsilon_{\alpha+1})$. %, where $\alpha^{\varepsilon +}$ is the least $\varepsilon$-number larger than $\alpha$.
In addition, $\mathsf{CT}$ derives $\alpha \in \pred{Lim} \to\{ [\forall \beta < \alpha I_0(\varepsilon_{\beta})] \to I_0(\varepsilon_{\alpha}) \}$.
These facts together mean that $I_0(\varepsilon_{x})$ is progressive, i.e., $\mathsf{CT} \vdash \mathrm{Prog} \lambda x I_0(\varepsilon_{x})$.
Since $\mathsf{CT}$ can derive the transfinite induction for  $I_0(\varepsilon_{x})$ up to $\varepsilon_0$, the schema $\mathrm{TI}(<\varepsilon_{\varepsilon_0})$ is obtained in $\mathsf{CT}$, as required.
To emulate this proof within $\mathsf{CR}$, we must extract the computational content of the proof.
For example, it is necessary to explicitly give a partial recursive function that computes a realiser of $\mathrm{TI}(\alpha+1)$
from that of $\mathrm{TI}(\alpha)$.

We define $I_0(e,\alpha) = \forall \ulcorner A \urcorner \in \pred{Sent}_{\mathcal{L}}. \ (e \cdot \ulcorner A \urcorner) \T \ulcorner \mathrm{TI}(A, \dot{\alpha})\urcorner$,
which means that $e$ computes a realiser of the transfinite induction $\mathrm{TI}(A, \alpha)$ for each $\mathcal{L}$-sentence $A$.

\begin{lemma}\label{well-ordering_basic}
\begin{enumerate}
\item There exists a number $\mathsf{k}_{0}$ such that:
\[
\mathsf{CR} \vdash I_0(\mathsf{k}_{0},0).
\]

\item There exists a number $\mathsf{k}_{suc}$ such that:
\[
\mathsf{CR} \vdash I_0(e,\alpha) \to I_0(\mathsf{k}_{suc} \cdot \langle e,\alpha \rangle, \alpha+1).
\]

\item There exists a number $\mathsf{k}_{\omega}$ such that:
\[
\mathsf{CR} \vdash I_0(e,\alpha) \to I_0(\mathsf{k}_{\omega} \cdot \langle e,\alpha \rangle, \omega^{\alpha}).
\]

\item There exists a number $\mathsf{k}_{lim}$ such that:
\[
\mathsf{CR} \vdash \alpha \in \pred{Lim} \to [\forall n I_0(e \cdot n,[\alpha]_n) \to I_0(\mathsf{k}_{lim} \cdot \langle e,\alpha \rangle, \alpha)].
\]
%\item There exists a primitive recursive function $f_{ord}$ such that:
%\[
%\mathsf{RCT} \vdash \mathsf{Prog} \lambda x I_0(f_{ord}(x),x)
%\forall \beta < \alpha I_0(f_{ord}(\beta),\beta) \to I_0(f_{ord}(\alpha), \alpha).
%\]
\end{enumerate}
%Note that the function $f_{suc}$ is defined depending on $\alpha,$ though we omit it for simplicity.
%The same thing is said to $f_{\omega}$ and $f_{lim},$ too.
\end{lemma}

%From the above explanation, the meaning of these functions should be clear. 
Note that $\mathsf{k}_0$ realises the schema $\mathrm{TI}(0)$; $\mathsf{k}_{suc}$ realises $\mathrm{TI}(\alpha + 1)$ from $\mathrm{TI}(\alpha)$; $\mathsf{k}_{\omega}$ realises $\mathrm{TI}(\omega^{\alpha})$ from $\mathrm{TI}(\alpha)$; $\mathsf{k}_{lim}$ realises $\mathrm{TI}(\alpha)$ for a limit ordinal $\alpha$ if each $\mathrm{TI}([\alpha]_n)$ is realised.

\begin{proof}
\begin{enumerate}
\item For every $\mathcal{L}$-formula $A$, we can primitive recursively find a proof of $\mathrm{TI}(A, 0)\urcorner)$ in $\mathsf{PA}$. Thus, by Theorem~\ref{formal_realisability_PA}, there exists a required $\mathsf{k}_{0}$.

\item In $\mathsf{PA}$, we can primitive recursively find a proof of $\forall \alpha (\mathrm{TI}(A, \alpha) \to \mathrm{TI}(A, \alpha+1))$ for each $\mathcal{L}$-formula $A$.
Thus, by Theorem~\ref{formal_realisability_PA}, there exists a number $\mathsf{k}$ such that $\mathsf{CR} \vdash ( (\mathsf{k} \cdot \alpha) \cdot \ulcorner A \urcorner) \T \ulcorner \mathrm{TI}(A, \alpha) \to \mathrm{TI}(A, \alpha+1) \urcorner.$
For this $\mathsf{k}$, we define $\mathsf{k}_{suc}$ to be such that for any $\mathcal{L}$-sentence $A$,
\[ ( \mathsf{k}_{suc} \cdot \langle e,\alpha \rangle) \cdot \ulcorner A \urcorner \simeq \mathsf{i} \cdot \langle (\mathsf{k} \cdot \alpha) \cdot \ulcorner A \urcorner, e \cdot \ulcorner A \urcorner \rangle.\] 
Then, if $e$ satisfies $I_0(e,\alpha),$ we have $(( \mathsf{k}_{suc} \cdot \langle e,\alpha \rangle) \cdot \ulcorner A \urcorner) \T \ulcorner \mathrm{TI}(A, \dot{\alpha} + 1) \urcorner$, as required.

\item For every $\mathcal{L}$-formula $A$, there exists an $\mathcal{L}$-formula $A'$ such that we can primitive recursively find a proof of $\forall \alpha (\mathrm{TI}(A', \alpha) \to \mathrm{TI}(A, \omega^{\alpha}))$ in $\mathsf{PA}$.
Thus, similar to the proof of the item $2$, there is a required function $\mathsf{k}_{\omega}$.

\item We assume $\forall n I_0(e\cdot n,[\alpha]_n).$ 
From this $e,$ we can primitive recursively define $\mathsf{k}^{\dag}$ such that:
\[
\forall \ulcorner A \urcorner \in \pred{Sent}_{\mathcal{L}}. \ (( \mathsf{k}^{\dag} \cdot \langle e,\alpha \rangle) \cdot \ulcorner A \urcorner) \T \ulcorner \forall \beta < \dot{\alpha} \mathrm{TI}(A, \beta) \urcorner.
\]
In addition, for each $\mathcal{L}$-formula $A$, we obtain $ \forall \beta < \alpha \mathrm{TI}(A, \beta) \to \mathrm{TI}(A, \alpha)$ in $\mathsf{PA}$; thus, according to Thorem~\ref{formal_realisability_PA}, we take a number $\mathsf{k}^{\ddag}$ such that: 
\[((\mathsf{k}^{\ddag} \cdot \alpha)\cdot \ulcorner A \urcorner) \T \ulcorner \forall \beta < \dot{\alpha} \mathrm{TI}(A, \beta) \to \mathrm{TI}(A, \dot{\alpha}) \urcorner.\]
Now, we define $\mathsf{k}_{lim}$ such that: 
\[( \mathsf{k}_{lim} \cdot \langle e,\alpha \rangle) \cdot \ulcorner A \urcorner \simeq \mathsf{i} \cdot \langle(\mathsf{k}^{\ddag} \cdot \alpha)\cdot \ulcorner A \urcorner, (\mathsf{k}^{\dag} \cdot \langle e,\alpha \rangle) \cdot \ulcorner A \urcorner \rangle .\]
We then obtain $(( \mathsf{k}_{lim} \cdot \langle e,\alpha \rangle) \cdot \ulcorner A \urcorner) \T \ulcorner \mathrm{TI}(A, \dot{\alpha}) \urcorner$, as required. \qed
\end{enumerate}
\end{proof}

The following lemma shows the progressiveness of the epsilon function.
\begin{lemma}\label{well-ordering_epsilon}
\begin{enumerate}
\item There exists a number $\mathsf{k}_{\varepsilon \_ 0}$ such that:
\[
\mathsf{CR} \vdash I_0(\mathsf{k}_{\varepsilon \_ 0},\varepsilon_0).
\]

\item There exists a number $\mathsf{k}_{\varepsilon \_ suc}$ such that:
\[
\mathsf{CR} \vdash I_0(e,\varepsilon_{\alpha}) \to I_0(\mathsf{k}_{\varepsilon \_ suc} \cdot \langle e,\alpha \rangle, \varepsilon_{\alpha+1}).
\]

\item There exists a number $\mathsf{k}_{\varepsilon} $ such that:
\[
\mathsf{CR} \vdash \mathrm{Prog}\lambda \alpha I_0(\mathsf{k}_{\varepsilon}\cdot \alpha, \varepsilon_{\alpha}).
\]
%\item There exists a primitive recursive function $f_{ord}$ such that:
%\[
%\mathsf{RCT} \vdash \mathsf{Prog} \lambda x I_0(f_{ord}(x),x)
%\forall \beta < \alpha I_0(f_{ord}(\beta),\beta) \to I_0(f_{ord}(\alpha), \alpha).
%\]
\end{enumerate}
\end{lemma}

\begin{proof}
\begin{enumerate}
\item We define a number $\mathsf{k}$ as follows:

$\begin{cases} 
\mathsf{k} \cdot0 :\simeq \mathsf{k}_{suc}\cdot \langle \mathsf{k}_0, 0\rangle, \\
\mathsf{k}\cdot(n+1) :\simeq \mathsf{k}_{\omega} \cdot \langle \mathsf{k}\cdot n, [\varepsilon_0]_n) & \text{for each $n$}.
\end{cases}$

Then, we clearly have $\mathsf{CR} \vdash \forall n (I_0(\mathsf{k} \cdot n, [\varepsilon_0]_n)$,
hence by item $4$ in Lemma~\ref{well-ordering_basic}, the number $\mathsf{k}_{\varepsilon \_ 0} := \mathsf{k}_{lim} \cdot \langle \mathsf{k}, \varepsilon_0\rangle$ is the required one.
%Note that when $\{f\}$ is primitive recursive, then so is $\{f_{lim} \cdot \langle f, \varepsilon_0\rangle \}.$

\item The proof is nearly the same as that of item $1$.

\item For an ordinal $\alpha$ and a number $e$, let $e[\alpha]$ be such that $e[\alpha]\cdot n :\simeq e\cdot [\alpha]_n.$
Here, $\mathsf{k}_{\varepsilon}$ is defined as follows:

$\begin{cases} 
 \mathsf{k}_{\varepsilon}\cdot 0 :\simeq \mathsf{k}_{\varepsilon\_0} \\
\mathsf{k}_{\varepsilon}\cdot (\alpha+1) :\simeq \mathsf{k}_{\varepsilon\_suc}\cdot \langle \mathsf{k}_{\varepsilon}\cdot \alpha, \alpha\rangle \\
\mathsf{k}_{\varepsilon}\cdot \alpha :\simeq \mathsf{k}_{lim} \cdot \langle \mathsf{k}_{\varepsilon}[\alpha], \varepsilon_{\alpha}\rangle & \text{for $\alpha \in \pred{Lim}.$}
\end{cases}$

Then, to show the claim,%$\mathsf{RCT} \vdash \mathsf{Prog}\lambda \alpha I_0(f_{\varepsilon}(\alpha), \varepsilon_{\alpha}),$
we take any $\alpha$ and assume $\forall \beta < \alpha  I_0(\mathsf{k}_{\varepsilon}\cdot \beta, \varepsilon_{\beta}).$
\begin{itemize}
\item If $\alpha = 0,$ then the conclusion $I_0(\mathsf{k}_{\varepsilon}\cdot \alpha, \varepsilon_{\alpha})$ is clear by item~1.
\item If $\alpha \in \pred{Suc},$ then by the assumption, we obtain $I_0(\mathsf{k}_{\varepsilon}\cdot (\alpha-1), \varepsilon_{\alpha-1})$;
thus, we obtain the conclusion by item~2.
\item If $\alpha \in \pred{Lim},$ then we obtain $I_0(\mathsf{k}_{\varepsilon}[\alpha]\cdot n, \varepsilon_{[\alpha]_n})$ for each $n$.
Here, since $[\varepsilon_{\alpha}]_n = \varepsilon_{[\alpha]_n},$ Lemma~\ref{well-ordering_basic} implies that $I_0(\mathsf{k}_{lim}\cdot \langle \mathsf{k}_{\varepsilon}[\alpha], \varepsilon_{\alpha} \rangle, \varepsilon_{\alpha}\rangle$; thus, it follows that $I_0(\mathsf{k}_{\varepsilon}\cdot \alpha, \varepsilon_{\alpha})$. \qed
\end{itemize}
\end{enumerate}
\end{proof}

\begin{lemma}\label{well-ordering_proof_CR}
For each $\mathcal{L}$-formula $A$ and for each ordinal number $\alpha < \varepsilon_{\varepsilon_0}$, we can find a term $s$ such that $\mathsf{CR} \vdash s \T \ulcorner \mathrm{TI}(\alpha, A) \urcorner.$
\end{lemma}

%We now prove Lemma~\ref{well-ordering_proof_CR}, as promised.
%\spnewtheorem*{recall:well-ordering_proof_CR}{Proof of Lemma~5}%\ref{well-ordering_proof_CR}}{\bfseries}{}%{\itshape}
\begin{proof}%\label{well-ordering_proof_CR}
%For each $\mathcal{L}$-formula $A$ and for each ordinal number $\alpha < \varepsilon_{\varepsilon_0}$, we can find a term $s$ such that $\mathsf{CR} \vdash s \mathrm{T} \ulcorner \mathsf{TI}(\alpha, A) \urcorner.$
%\end{recall:well-ordering_proof_CR}
%\begin{proof}
We fix any ordinal $\alpha < \varepsilon_{\varepsilon_0},$ and then there exists an ordinal $\beta < \varepsilon_0$ such that $\alpha \leq \varepsilon_{\beta} < \varepsilon_{\varepsilon_0}.$
Thus, by taking any $\mathcal{L}$-formula $A,$ it is sufficient to prove $ ((\mathsf{k}_{\varepsilon}\cdot \beta ) \cdot \ulcorner A \urcorner) \T \ulcorner \mathrm{TI}(\varepsilon_{\beta}, A) \urcorner$
because we can primitive recursively find a required term $s$ from the term $(\mathsf{k}_{\varepsilon}\cdot \beta ) \cdot \ulcorner A \urcorner$.

Since $\mathsf{PA}$ derives any transfinite induction for $\beta < \varepsilon_0,$ we have $\mathrm{TI}(\beta, I_0(\mathsf{k}_{\varepsilon}\cdot x, \varepsilon_x))$.
Therefore, according to item~3 in Lemma~\ref{well-ordering_epsilon}, we have $I_0(\mathsf{k}_{\varepsilon}\cdot \beta, \varepsilon_{\beta})$.
Thus, it follows that $((\mathsf{k}_{\varepsilon}\cdot \beta) \cdot \ulcorner A \urcorner) \T \ulcorner \mathrm{TI}(\varepsilon_{\beta}, A) \urcorner,$ as desired. \qed
%Finally, depending on $A$ and $\alpha$, define a closed term $s$ to be:
%\[
%s = \lambda b. \langle (f_{\varepsilon}\cdot \alpha) \cdot \ulcorner A \urcorner,  b \rangle. 
%\]
% Then, by the axiom $(\mathsf{Ax}_{\pole}),$ we obtain $\langle s, b \rangle \in \pole$ for any $b \mathrm{F} \ulcorner \mathsf{TI}(\varepsilon_{\alpha}, A) \urcorner$.
%Therefore, it follows that $\mathsf{RCT} \vdash s \mathrm{T} \ulcorner \mathsf{TI}(\varepsilon_{\alpha}, A) \urcorner.$
%Note that by the construction, $\{ \{f_{\varepsilon}\}(\alpha) \} (\ulcorner A \urcorner)$ has a value. 
\end{proof}
%\end{well-ordering_proof_CR}

By Lemma~\ref{well-ordering_proof_CR} and the reflection rule,  $\mathsf{CR}^+$ yields the formula $\mathrm{TI}(\alpha, A)$
for each $\mathcal{L}$-formula $A$ and  $\alpha < \varepsilon_{\varepsilon_0}$.
Thus, $\mathsf{CR}^+$ derives every theorem of $\mathsf{PA} + \mathsf{TI}({<} \varepsilon_{\varepsilon_0})$.
Therefore, according to Proposition~\ref{prop:CT_and_CR_with_empty} and Theorem~\ref{fact:strength_CT}, the proof of Theorem~\ref{thm:strength_CR+} is completed.

%%%%%%%%%%%%%%%%%%%%%%%%%%%%%%%%%%%%%%%%%%%%%%%%

\section{Ramified Realisability}\label{sec:ramified_realisability}

In the previous sections, we formulate a typed theory $\mathsf{CR}$ of classical relisability and study its proof-theoretic properties.
Moreover, we show that an extension $\mathsf{CR}^{\emptyset}$ is closely related to the typed theory $\mathsf{CT}$ of truth.
In formal truth theory, various generalisations of $\mathsf{CT}$ have been proposed by authors.
It is therefore expected to formulate theories of realisability corresponding to such truth theories.

One natural direction is to remove type restriction on $\mathsf{CR}$, that is, to give falsification/realisation conditions also to sentences that contain falsification or realisation predicates.
Such an approach lead, in~\cite{hayashi2025friedman}, to a theory $\mathsf{FSR}$, corresponding to the well-known truth theory $\mathsf{FS}$ by Friedman and Sheard \cite{friedman1987axiomatic}.
Although $\mathsf{FSR}$ has all of the compositional axioms for $\mathcal{L}_R$-sentences, it is difficult to strengthen the theory further because $\mathsf{FSR}$ almost carries over $\omega$-inconsistency, an undesirable property of $\mathsf{FS}$.
To be more precise, the theory obtained by adding to $\mathsf{FSR}$ the reflection principle for $\mathsf{FSR}$ derives $\pole \neq \emptyset$, which means, by Lemma~\ref{lem:continuation_constant}, that false sentences are realisable~\cite[Proposition~12]{hayashi2025friedman}.

Another way of generalisation is the hierarchical approach, which is to add meta truth predicates repeatedly.
In particular, the well-known theory $\mathsf{RT}_{< \gamma}$ (ramified truth) has truth predicates: $\T_0, T_1, \dots, T_{\beta}, \dots$ ($\beta < \gamma$), each of which has the compositional axioms for sentences containing only truth predicates with lower indices. 
Here $\gamma$ can be increased up to transfinite ordinals, thereby stronger theories are obtained. 

The purpose of this section is then to formulate theories $\mathsf{RR}_{< \gamma}$ of ramified realisability, corresponding to $\mathsf{RT}_{< \gamma}$.
Moreover, we generalise the main results of previous sections to those for $\mathsf{RR}_{< \gamma}$.

Firstly, we introduce $\mathsf{RT}_{< \gamma}$.
Taking any $\gamma < \pred{OT}$,
we define $\Lt^{< \gamma} := \mathcal{L} \cup \{ \T_{\beta} \ | \ \beta < \gamma \}$.
We now fix some G\"{o}del-numbering for $\Lt^{< \gamma}$.
Then, $\pred{Sent}_{\mathrm{R}}^{< \gamma}(x)$ means that $x$ is the code of an $\Lt^{< \gamma}$-sentence.
\begin{definition}[cf.~{\cite[Definition~9.2]{halbach2014axiomatic}}]\label{defn:RT}
%For a unary predicate $\mathrm{T},$ let $\mathcal{L}_{\mathrm{T}} = \mathcal{L} \cup \{ \mathrm{T} \}.$
The $\Lt^{< \gamma}$-theory $\mathsf{RT}_{< \gamma}$ (ramified truth) consists of $\mathsf{PA}$ formulated for the language $\Lt^{< \gamma}$ plus the following three axioms for each $\alpha < \beta < \gamma$.
\begin{description}

\item[$(\mathrm{RT}1_{\beta})$] $\forall \ulcorner s \urcorner, \ulcorner t \urcorner. \ \forall \ulcorner A_x \urcorner \in \pred{Sent}_{\mathrm{T}}^{< \beta}. \ \pred{Eq}( \ulcorner s \urcorner,\ulcorner t \urcorner) \to ( \T_{\beta} \ulcorner A(s) \urcorner \to \T_{\beta} \ulcorner A(t) \urcorner )$.

\item[$(\mathrm{RT}2_{\beta})$] $\T_{\beta} \ulcorner P(\dot{\vec{x}}) \urcorner  \leftrightarrow P(\vec{x}) $, for each atomic predicate $P(\vec{x})$ of $\mathcal{L}$.

\item[$(\mathrm{RT}3_{\beta})$] $\forall \gn {A },\gn {B} \in \pred{Sent}_{\mathrm{T}}^{<\beta}. \ \T_{\beta}\gn {A \to B} \leftrightarrow ( \T_{\beta} \gn A \to \T_{\beta}\gn B ) $.

\item[$(\mathrm{RT}4_{\beta})$] $\forall \gn {A_v} \in \pred{Sent}_{\mathrm{T}}^{<\beta}. \ \T_{\beta}\gn {\forall v A} \leftrightarrow \forall x \T_{\beta} \gn {A(\dot{x})} $.

\item[$(\mathrm{RT}5_{\beta})$] $\forall \ulcorner A \urcorner \in \pred{Sent}_{\mathrm{T}}^{<\alpha}. \ \T_{\beta} \ulcorner \T_{\alpha} \ulcorner A \urcorner \urcorner \leftrightarrow \T_{\alpha} \ulcorner A \urcorner$.

\item[$(\mathrm{RT}6_{\beta})$] $\forall \delta < \beta ( \forall \ulcorner A \urcorner \in \pred{Sent}_{\mathrm{T}}^{<\delta}. \ \T_{\beta} \ulcorner \T_{\delta} \ulcorner A \urcorner \urcorner \leftrightarrow \T_{\beta} \ulcorner A \urcorner )$.

\end{description}
\end{definition}

In the above definition, $(\mathrm{RT}1_{\beta})$ is just term invariance (cf.~Proposition~\ref{term-regular});
$(\mathrm{RT}2_{\beta})$ to $(\mathrm{RT}4_{\beta})$ are hierarchical generalisations of the compositional axioms of $\mathsf{CT}$ (Definition~\ref{defn:CT});
$(\mathrm{RT}5_{\beta})$ expresses Tarski-biconditional for $\T_{\alpha} \ulcorner A \urcorner$;
$(\mathrm{RT}6_{\beta})$ means that Tarski-biconditional for $\mathcal{L}_{\mathrm{T}}^{< \delta}$ holds inside $\T_{\beta}$.

Since $(\mathrm{RT}1_{\beta})$ to $(\mathrm{RT}4_{\beta})$ are just hierarchical generalisations of the axioms of $\mathsf{CT}$, it is straightworward to give their realisabilitistic counterparts.
As for $(\mathrm{RT}5_{\beta})$ and $(\mathrm{RT}6_{\beta})$, we need to paraphrase Tarski-biconditional in terms of realisability.
As Tarski-biconditional tells equivalence between formal truth $\T_{\beta} \ulcorner A \urcorner$ and explicit truth $A$, one natural idea is to formulate explicit realisation $x \in \lvert A \rvert$ as a particular formula.
That way, Tarski-biconditional can be understood as the equivalence between formal realisation $x \T_{\beta} \ulcorner A \urcorner$ and explicit realisation $x \in \lvert A \rvert$.
So, following \cite{hayashi2025friedman}, we first define explicit falsification $x \in \lVert A \rVert$, from which explicit realisation $x \in \lvert A \rvert$ is also defined.

Let $\mathcal{L}_{\pole} = \mathcal{L} \cup \{ x \in \pole \}$.
Taking any ordinal $\gamma \in \pred{OT}$, we define $\mathcal{L}_{\mathrm{R}}^{< \gamma} := \mathcal{L}_{\pole}
\cup \{ x\F_{\beta}y \ | \ \beta < \gamma \}
\cup \{ x\T_{\beta}y \ | \ \beta < \gamma \}$.
We now fix some G\"{o}del-numbering for $\mathcal{L}_{\mathrm{R}}^{< \gamma}$.

\begin{definition}[cf.~{\cite[Definition~9]{hayashi2025friedman}}]\label{defn:emplicit_realisability}
For each term $s$ and  
$\Lr^{< \gamma} \mhyph$formula $A$,
we inductively define $\Lr^{< \gamma}$-formulas $s \in \lvert A \rvert$ and $s \in \lVert A \rVert,$
with renaming bound variables in $A$ if necessary.
%\leigh{L: Again, \( = \) instead of \( \leftrightarrow \)}
\begin{itemize}
\item $s \in \lVert P \vec x \rVert  \ = \ P \vec x \to s \in \pole,$ if $P \in \mathcal{L}_{\pole}$;
\item $s \in \lVert t \F_{\beta} u \rVert  \ = \  %Sent_{\mathcal{L}^{+}}(u)\land 
s \F_{\beta}(t\dot{\in} \lVert u \rVert)$, for $\beta < \gamma$;
\item $s \in \lVert t \T_{\beta} u \rVert  \ = \  %Sent_{\mathcal{L}^{+}}(u) \land 
s \F_{\beta}(t\dot{\in} \lvert u \rvert)$, for $\beta < \gamma$;
\item $s \in \lVert A \to B \rVert  \ = \  (s)_{0} \in \lvert A \rvert \land (s)_{1} \in \lVert B \rVert;$
\item $s \in \lVert \forall x A(x) \rVert  \ = \  (s)_1 \in \lVert A((s)_0) \rVert;$
\item $s \in \lvert A \rvert  \ = \  \forall a(a \in \lVert A \rVert \to \langle s, a \rangle \in \pole)$ for a fresh variable $a$.
\end{itemize}
Here, $x \dot{\in} \lvert y \rvert$ is a binary primitive recursive function symbol such that $\mathsf{PA}$ derives $x \dot{\in} |\ulcorner A \urcorner| = \ulcorner \dot{x} \in A \urcorner$ for any $\Lr^{< \gamma}$-sentence $A$.
The existence of such a function is ensured by the primitive recursion theorem. %(for a proof, see, e.g., \cite{}). 
Moreover, we define such a symbol $x \dot{\in} \lvert y \rvert$ so that it does not return the code of any $\Lr^{< \gamma}$-sentence when $y$ is not. 
The other binary primitive recursive function symbol $x \dot{\in} \lVert y \rVert$ is similarly defined.
\end{definition}

%The next lemma shows that explicit realisability and formal realisability are equivalent for $\mathcal{L}_{\pole}$.
%\begin{lemma}[cf.~{\cite[Lemma~9]{ }}]\label{explicit_formal_realisability}
%\begin{enumerate}
%	\item Let $P \vec{y} $ be an atomic formula of $\mathcal{L}_{\pole}.$
%		\begin{enumerate}
%			\item $\mathsf{PA} \vdash x \F \ulcorner P \dot{\vec{y}} \urcorner \leftrightarrow x \in \lVert			
%			P\vec{y} \rVert.$
%			\item $\mathsf{PA} \vdash x \T \ulcorner P \dot{\vec{y}} \urcorner \leftrightarrow x \in \lvert 
%			P \vec{y} \rvert.$
%		\end{enumerate}
%		
%	\item For every terms $s,t$, we have:
%		\begin{enumerate}
%			\item $\mathsf{PA} \vdash x \in \lvert s \F t \rvert \leftrightarrow x \T (s \dot{\in} \lVert t \rVert),$
%			\item $\mathsf{PA} \vdash x \in \lvert s \T t \rvert \leftrightarrow x \T (s \dot{\in} \lvert t \rvert).$
%		\end{enumerate}
%		
%\end{enumerate}
%\end{lemma}

Based on the above definition of explicit falsification and realisation, we now formualate a counterpart $\mathsf{RR}_{< \gamma}$ of $\mathsf{RT}_{< \gamma}$.

Let $\pred{Sent}_{\mathrm{R}}^{< \gamma}(x)$ mean that $x$ is the code of an $\Lt^{< \gamma}$-sentence.
\begin{definition}\label{defn:RR}

The $\Lr^{< \gamma}$-theory $\mathsf{RR}_{< \gamma}$ (ramified realisability) consists of $\mathsf{PA}$ formulated for the language $\Lr^{< \gamma}$ plus the following three axioms for each $\alpha < \beta < \gamma$.

\begin{description}
\item[$(\mathrm{RR}1_{\beta})$] $\pred{T}_1(a,b,c) \to ( \pred (c)_0 \in \pole \to \langle a, b \rangle \in \pole)$.

\item[$(\mathrm{RR}2_{\beta})$] $\forall \ulcorner A \urcorner \in \pred{Sent}_{\mathrm{R}}^{< \beta}. \ a \T_{\beta} \ulcorner A \urcorner \leftrightarrow \forall b ( b \F_{\beta} \ulcorner A \urcorner \to \langle a, b \rangle \in \pole )$.

\item[$(\mathrm{RR}3_{\beta})$] $\forall \ulcorner s \urcorner, \ulcorner t \urcorner. \ \forall \ulcorner A_x \urcorner \in \pred{Sent}_{\mathrm{R}}^{< \beta}. \ \pred{Eq}( \ulcorner s \urcorner,\ulcorner t \urcorner) \to ( a \F_{\beta} \ulcorner A(s) \urcorner \to a \F_{\beta} \ulcorner A(t) \urcorner )$.

\item[$(\mathrm{RR}4_{\beta})$] $a \F_{\beta}\ulcorner P \dot{\vec{x}} \urcorner \leftrightarrow (P(\vec{x}) \to \vec{x} \in \pole)$, for each atomic predicate $P(\vec{x})$ of $\mathcal{L}_{\pole}$.

%\item[$(\mathrm{GCR}_{P2})$] $P \vec{x} \to (a \F\ulcorner P \dot{\vec{x}} \urcorner \leftrightarrow a \in \pole)$

\item[$(\mathrm{RR}5_{\beta})$] $\forall \ulcorner A \urcorner, \ulcorner B \urcorner \in \pred{Sent}_{\mathrm{R}}^{< \beta}. \ a \F_{\beta}\ulcorner A \to B \urcorner \leftrightarrow \big( (a)_0 \T_{\beta} \ulcorner A \urcorner \land (a)_1\F_{\beta}\ulcorner B \urcorner \big)$.

\item[$(\mathrm{RR}6_{\beta})$] $\forall \ulcorner A_x \urcorner \in \pred{Sent}_{\mathrm{R}}^{< \beta}. \ a \F_{\beta}\ulcorner \forall x A \urcorner \leftrightarrow (a)_1 \F_{\beta} \ulcorner A((\dot{a})_0) \urcorner $.

\item[$(\mathrm{RR}7_{\beta})$] $\forall \ulcorner A \urcorner \in \pred{Sent}_{\mathrm{R}}^{<\alpha}. \ a\F_{\beta} \ulcorner \dot{b} \F_{\alpha} \ulcorner A \urcorner \urcorner \leftrightarrow a \in \lVert b\F_{\alpha} \ulcorner A \urcorner \rVert$.

\item[$(\mathrm{RR}8_{\beta})$] $\forall \ulcorner A \urcorner \in \pred{Sent}_{\mathrm{R}}^{<\alpha}. \ a\F_{\beta} \ulcorner \dot{b} \T_{\alpha} \ulcorner A \urcorner \urcorner \leftrightarrow a \in \lVert b\T_{\alpha} \ulcorner A \urcorner \rVert$.

\item[$(\mathrm{RR}9_{\beta})$] $\forall \delta < \beta ( \forall \ulcorner A \urcorner \in \pred{Sent}_{\mathrm{R}}^{<\delta}. \ a\F_{\beta} \ulcorner \dot{b} \F_{\delta} \ulcorner A \urcorner \urcorner \leftrightarrow a\F_{\beta} \ulcorner \dot{b} \in \lVert A \rVert \urcorner )$.

\item[$(\mathrm{RR}10_{\beta})$] $\forall \delta < \beta ( \forall \ulcorner A \urcorner \in \pred{Sent}_{\mathrm{R}}^{<\delta}. \ a\F_{\beta} \ulcorner \dot{b} \T_{\delta} \ulcorner A \urcorner \urcorner \leftrightarrow a\F_{\beta} \ulcorner \dot{b} \in \lvert A \rvert \urcorner )$.

%\item[$(\mathrm{RR}11_{\beta})$] $a \T_{\beta} b \to \pred{Sent}^{< \beta}_R (b)$.
\end{description}

$\mathsf{RR}_{\gamma}$ is then defined as $\mathsf{RR}_{< \gamma + 1}$.
Similar to Definition~\ref{defn:reflection_rule}, let $\mathsf{RR}_{<\gamma}^{+}$ be 
$\mathsf{RR}_{<\gamma}$ augmented with the following reflection rule:
\begin{center} \
\infer[]{A}{t \T_{\beta}\ulcorner A \urcorner} for $A \in \mathcal{L}$ and $\beta < \gamma$.
\end{center}
Similarly, $\mathsf{RR}_{<\gamma}^{\emptyset}$ is defined to be $\mathsf{RR}_{<\gamma} + \pole = \emptyset$.
\end{definition}

We can easily verify the following:
\begin{corollary}[cf.~{\cite[Lemma~9]{hayashi2025friedman}}]\label{cor:explicit_realisability}
The following are derivable for each $\beta < \gamma$ in $\mathsf{RR}_{< \gamma}$:
\begin{itemize}
\item $\pred{Sent}_{\pred{R}}^{< \beta}(u) \to \bigl( s \in \lvert t \F_{\beta} u \rvert  \leftrightarrow 
s \T_{\beta}(t\dot{\in} \lVert u \rVert) \bigr)$
\item $\pred{Sent}_{\pred{R}}^{< \beta}(u) \to \bigl( s \in \lvert t \T_{\beta} u \rvert  \leftrightarrow  
s \T_{\beta}(t\dot{\in} \lvert u \rvert) \bigr)$
\end{itemize}
\end{corollary}

\begin{corollary}\label{cor:RR}
The following are derived for each $\alpha < \beta < \gamma$ in $\mathsf{RR}_{< \gamma}$:
\begin{description}
\item[$(\mathrm{RR}3_{\beta})'$] $\forall \ulcorner s \urcorner, \ulcorner t \urcorner. \ \forall \ulcorner A_x \urcorner \in \pred{Sent}_{\mathrm{R}}^{< \beta}. \ \pred{Eq}( \ulcorner s \urcorner,\ulcorner t \urcorner) \to ( a \T_{\beta} \ulcorner A(s) \urcorner \to a \T_{\beta} \ulcorner A(t) \urcorner )$.

\item[$(\mathrm{RR}7_{\beta})'$] $\forall \ulcorner A \urcorner \in \pred{Sent}_{\mathrm{R}}^{<\alpha}. \ a\T_{\beta} \ulcorner \dot{b} \F_{\alpha} \ulcorner A \urcorner \urcorner \leftrightarrow a \in \lvert b\F_{\alpha} \ulcorner A \urcorner \rvert$.

\item[$(\mathrm{RR}8_{\beta})'$] $\forall \ulcorner A \urcorner \in \pred{Sent}_{\mathrm{R}}^{<\alpha}. \ a\T_{\beta} \ulcorner \dot{b} \T_{\alpha} \ulcorner A \urcorner \urcorner \leftrightarrow a \in \lvert b\T_{\alpha} \ulcorner A \urcorner \rvert$.

\item[$(\mathrm{RR}9_{\beta})'$] $\forall \delta < \beta ( \forall \ulcorner A \urcorner \in \pred{Sent}_{\mathrm{R}}^{<\delta}. \ a\T_{\beta} \ulcorner \dot{b} \F_{\delta} \ulcorner A \urcorner \urcorner \leftrightarrow a\T_{\beta} \ulcorner \dot{b} \in \lVert A \rVert \urcorner )$.

\item[$(\mathrm{RR}10_{\beta})'$] $\forall \delta < \beta ( \forall \ulcorner A \urcorner \in \pred{Sent}_{\mathrm{R}}^{<\delta}. \ a\T_{\beta} \ulcorner \dot{b} \T_{\delta} \ulcorner A \urcorner \urcorner \leftrightarrow a\T_{\beta} \ulcorner \dot{b} \in \lvert A \rvert \urcorner )$.
\end{description}
\end{corollary}

\begin{proof}
We consider only $(\mathrm{RR}7_{\beta})'$, as the other cases are similarly proved.
Thus, taking any $\ulcorner A \urcorner \in \pred{Sent}_{\mathrm{R}}^{<\delta}$, we derive $a\T_{\beta} \ulcorner \dot{b} \F_{\alpha} \ulcorner A \urcorner \urcorner \leftrightarrow a \in \lvert b\F_{\alpha} \ulcorner A \urcorner \rvert$ as follows:
\begin{align}
a\T_{\beta} \ulcorner \dot{b} \F_{\alpha} \ulcorner A \urcorner \urcorner &\Leftrightarrow \forall c (c \F_{\beta} \ulcorner \dot{b} \F_{\alpha} \ulcorner A \urcorner \urcorner \to \langle a,c \rangle \in \pole) &&\text{by $\mathrm{RR}2_{\beta}$} \notag \\
&\Leftrightarrow \forall c (c \in \lVert b \F_{\alpha} \ulcorner A \urcorner \rVert \to \langle a,c \rangle \in \pole) &&\text{by $\mathrm{RR}7_{\beta}$} \notag \\
&\Leftrightarrow a \in \lvert b\F_{\alpha} \ulcorner A \urcorner \rvert &&\text{by Definition~\ref{defn:emplicit_realisability}} \notag
\end{align}
\qed
\end{proof}

%\begin{lemma}\label{lem:explicit_formal_RR}
%Take any ordinal $\gamma$.
%For any $\beta < \gamma$ and $\mathcal{L}_R^{< \beta}$-formula $A$, we have $\mathsf{RR}_{< \gamma} \vdash (a \F_{\beta} \ulcorner A \urcorner \leftrightarrow a \in \lVert A \rVert) \land (a \T_{\beta} \ulcorner A \urcorner \leftrightarrow a \in \lvert A \rvert)$.
%\end{lemma}

%For $S \in \{ \mathsf{RR}_{<\gamma}, \mathsf{RR}_{<\gamma}^+, \mathsf{RR}_{<\gamma}^{\emptyset} \}$, we write $S^{-}$ to mean the theory $S$ without the axiom schema $(\mathrm{RR}11_{\beta})$.

%\begin{lemma}\label{lem:conservative_R-normality}
%We define a translation $\mathcal{T}_{norm}: \mathcal{L}^{< \gamma}_R \to \mathcal{L}^{< \gamma}_R$ such that $\mathcal{T}_{norm}(s \T_{\beta} t) := s \T_{\beta} t \land \pred{Sent}^{< \beta}_{R}(t)$ for $\beta < \gamma$ and the other vocabularies are unchanged.

%Then, $S \vdash A$ implies $S^{-} \vdash \mathcal{T}_{norm}(A)$ for every $\mathcal{L}_R^{< \gamma}$-formula $A$.
%\end{lemma}

We first generalise the model theory for $\mathsf{CR}$ to for ramified theories.
Given an ordinal $\gamma$ and a pole $\pole$, the falsification relation $x \in \lVert A \rVert_{\pole}^{< \gamma}$ and the realisation relation $x \in \lvert A \rvert_{\pole}^{< \gamma}$ for $\mathcal{L}_{\mathrm{R}}^{< \gamma}$ are defined  inductively:
the vocabularies of $\mathcal{L}$ are interpreted exactly in the same way as in Definition~\ref{defn:classical_number_realisability};
the falsity condition for the predicates $x \in \pole$, $x \F_{\beta} y$, and $x \T_{\beta} y$ is given as follows.
\begin{align}
n \in \lVert m \in \pole \rVert_{\pole}^{< \gamma} \ &\Leftrightarrow \ \text{if} \ m \in \pole \ \text{then} \ n \in \pole \notag \\
n \in \lVert m \F_{\beta} l \rVert_{\pole}^{< \gamma} \ &\Leftrightarrow \ l = \ulcorner A \urcorner \ \& \ n \in \lVert m \in \lVert A \rVert \rVert_{\pole}^{< \gamma} \text{ for some $A \in \pred{Sent}_{\mathrm{R}}^{< \beta}$},  \notag \\
n \in \lVert m \T_{\beta} l \rVert_{\pole}^{< \gamma} \ &\Leftrightarrow \ l = \ulcorner A \urcorner \ \& \ n \in \lVert m \in \lvert A \rvert \rVert_{\pole}^{< \gamma} \text{ for some $A \in \pred{Sent}_{\mathrm{R}}^{< \beta}$}.  \notag
\end{align}

Here, we note that if $A \in \pred{Sent_R^{< \beta}}$, then so are $m \in \lVert A \rVert$ and $m \in \lvert A \rvert$.
Thus, the above definition is indeed an inductive definition. 

The interpretation of the predicates $x \F_{\beta} y$ and $x \T_{\beta} y$ is then given for each $\beta < \gamma$:
\begin{align}
\mathbb{F}_{\pole}^{\beta} &= \{ (n, \ulcorner A \urcorner) \in \mathbb{N}^2 \ | \ A \in \pred{Sent}_{\mathrm{R}}^{< \beta} \ \& \ n \in \lVert A \rVert_{\pole}^{< \gamma} \}, \notag \\
\mathbb{T}_{\pole}^{\beta} &= \{ (n, \ulcorner A \urcorner) \in \mathbb{N}^2 \ | \ A \in \pred{Sent}_{\mathrm{R}}^{< \beta} \ \& \ n \in \lvert A \rvert_{\pole}^{< \gamma} \}. \notag
\end{align}

The $\Lr^{< \gamma}$-model $\langle \mathbb{N}, \pole, \{ \mathbb{F}_{\pole}^{\beta} \}_{\beta < \gamma}, \{ \mathbb{T}_{\pole}^{\beta} \}_{\beta < \gamma} \rangle$ is thus obtained.
For simplicity, we will write $\langle \pole, < \gamma \rangle \models A$ instead of $\langle \mathbb{N}, \pole, \{ \mathbb{F}_{\pole}^{\beta} \}_{\beta < \gamma}, \{ \mathbb{T}_{\pole}^{\beta} \}_{\beta < \gamma} \rangle \models A$.

The next lemma shows that 
for every $A \in \pred{Sent}_{\mathrm{R}}^{< \beta}$,
formal falsification $s \F_{\beta} \ulcorner A \urcorner$ and
explicit falsification $x \in \lVert A \rVert$ are equivalent in the model;
the same applies to formal realisation $x \T_{\beta} \ulcorner A \urcorner$ and
explicit realisation $x \in \lvert A \rvert$.

\begin{lemma}\label{lem:explicit-realisability-in-model}
Take any $\beta < \gamma$.
Then, the following hold for any $\Lr^{< \beta}$-sentence $A$ and any number $a$.
\begin{enumerate}
\item $\langle \pole, < \gamma \rangle \models a \F_{\beta} \ulcorner A \urcorner$ holds if and only if $\langle \pole, < \gamma \rangle \models a \in \lVert A \rVert$.
\item $\langle \pole, < \gamma \rangle \models a \T_{\beta} \ulcorner A \urcorner$ holds if and only if $\langle \pole, < \gamma \rangle \models a \in \lvert A \rvert$.
\end{enumerate}
\end{lemma}

\begin{proof}
Firstly, Item~2 can be proved by Item~1:
\begin{align}
& \langle \pole, < \gamma \rangle \models a \T_{\beta} \ulcorner A \urcorner \notag \\
&\Leftrightarrow \ a \in \lvert A \rvert_{\pole}^{< \gamma} \notag \\ 
&\Leftrightarrow b \in \lVert A \rVert_{\pole}^{< \gamma} \text{ implies } \langle a,b \rangle \in \pole \text{ for any $b \in \mathbb{N}$} \notag \\
&\Leftrightarrow \langle \pole, < \gamma \rangle \models b \F_{\beta} \ulcorner A \urcorner \text{ implies } \langle a,b \rangle \in \pole \text{ for any $b \in \mathbb{N}$} \notag \\
&\Leftrightarrow \langle \pole, < \gamma \rangle \models b \in \lVert A \rVert \text{ implies } \langle a,b \rangle \in \pole \text{ for any $b \in \mathbb{N}$} &&\text{by Item~1} \notag \\
&\Leftrightarrow \langle \pole, < \gamma \rangle \models \forall b (b \in \lVert A \rVert \to \langle a,b \rangle \in \pole) \notag \\
&\Leftrightarrow \langle \pole, < \gamma \rangle \models a \in \lvert A \rvert \notag
\end{align}

Next, we prove Item~1 by induction on the logical complexity of $A$.
We divide the cases by the form of $A$:
\begin{description}
\item[($A = P$)] Assume that $A$ is of the form $P$  for some $\mathcal{L}_{\pole}$-atomic predicate $P$.
Then, 
\begin{align}
\langle \pole, < \gamma \rangle \models a \F_{\beta} \ulcorner P \urcorner
&\Leftrightarrow a \in \lVert P \rVert_{\pole}^{< \gamma} \notag \\
&\Leftrightarrow \text{if} \ P \ \text{then} \ a \in \pole \notag \\
&\Leftrightarrow \langle \pole, < \gamma \rangle \models P \to a \in \pole \notag \\
&\Leftrightarrow \langle \pole, < \gamma \rangle \models a \in \lVert P \rVert \notag
\end{align}

\item[($A = s \F_{\alpha} t$)] Assume that $A$ is of the form $s \F_{\alpha} t$ for some $\alpha < \beta$. Then,
\begin{align}
& \langle \pole, < \gamma \rangle \models a \F_{\beta} \ulcorner s \F_{\alpha} t \urcorner \notag \\ 
&\Leftrightarrow a \in \lVert s \F_{\alpha} t \rVert_{\pole}^{< \gamma} \notag \\ 
&\Leftrightarrow a \in \lVert s \in \lVert B \rVert \rVert_{\pole}^{< \gamma} \ \text{for some $\ulcorner B \urcorner = t \in \pred{Sent}_{\mathrm{R}}^{< \alpha}$} \notag \\
%&\Leftrightarrow a \in \lVert s \in \lVert B \rVert \rVert_{\pole}^{< \alpha} \ \text{for some $\ulcorner B \urcorner = t \in \pred{Sent}_R^{< \alpha}$} \notag \\
&\Leftrightarrow \langle \pole, < \gamma \rangle \models a \F_{\alpha} (s \dot{\in} \lVert t \rVert) \notag \\
&\Leftrightarrow \langle \pole, < \gamma \rangle \models a \in \lVert s \F_{\alpha} t \rVert \notag 
%&\Leftrightarrow \langle \pole, < \gamma \rangle \models a \in \lVert s \in \lVert B \rVert \rVert \ \text{for some $\ulcorner B \urcorner = t \in \pred{Sent}_R^{< \alpha}$} \notag \\
%&\Leftrightarrow \langle \pole, < \gamma \rangle \models a \F_{\beta} (s \dot{\in} t) \notag \\
%&\Leftrightarrow \langle \pole, < \gamma \rangle \models a \in \lVert s \F_{\beta} t \rVert \notag
\end{align}

\item[($A = s \T_{\alpha} t$)] Samely as above. 

\item[(Inductive step)] When $A$ is a complex sentence,  the claim immediately follows by the induction hypothesis. 
\end{description}
\qed
\end{proof}

Consequently, soundness of $\mathsf{RR}_{< \gamma}$
with respect to the model $\langle \pole, < \gamma \rangle$ is obtained for any pole.
\begin{proposition}\label{prop:soundness_RR}
	Let $A$ be any $\mathcal{L}_R^{< \gamma}$-sentence and assume that $\mathsf{RR}_{< \gamma}^+ \vdash A$.
Then, $\langle \pole, < \gamma \rangle \models A$ for any pole $\pole$.
\end{proposition}

\begin{proof}
The proof is by induction on the derivation of $A$.
If $A$ is an axiom of $\mathsf{PA}$ or one of the axioms $\mathrm{RR}1_{\beta}$ to $\mathrm{RR}6_{\beta}$, then the proof is almost the same as for Proposition~\ref{prop:soundness_CR+}.
The case of $\mathrm{RR}7_{\beta}$ or $\mathrm{RR}8_{\beta}$ is already dealt with in Lemma~\ref{lem:explicit-realisability-in-model}. 
 
As for $\mathrm{RR}9_{\beta}$, we take any $\delta < \beta$ and $\ulcorner A \urcorner \in \pred{Sent}_{\mathrm{R}}^{< \delta}$. 
Then, satisfaction of $a\F_{\beta} \ulcorner \dot{b} \F_{\delta} \ulcorner A \urcorner \urcorner \leftrightarrow a\F_{\beta} \ulcorner \dot{b} \in \lVert A \rVert \urcorner$ is proved as follows:
\begin{align}
\langle \pole, < \gamma \rangle \models a\F_{\beta} \ulcorner \dot{b} \F_{\delta} \ulcorner A \urcorner \urcorner &\Leftrightarrow \langle \pole, < \gamma \rangle \models a \in \lVert b \F_{\delta} \ulcorner A \urcorner \rVert &&\text{by Lemma~\ref{lem:explicit-realisability-in-model}} \notag \\
&\Leftrightarrow \langle \pole, < \gamma \rangle \models a \F_{\delta} \ulcorner \dot{b} \in \lVert A \rVert \urcorner &&\text{by Definition~\ref{defn:emplicit_realisability}} \notag \\
&\Leftrightarrow \langle \pole, < \gamma \rangle \models a \in \lVert b \in \lVert A \rVert \rVert &&\text{by Lemma~\ref{lem:explicit-realisability-in-model}} \notag \\
&\Leftrightarrow \langle \pole, < \gamma \rangle \models a \F_{\beta} \ulcorner \dot{b} \in \lVert A \rVert \urcorner &&\text{by Lemma~\ref{lem:explicit-realisability-in-model}} \notag
\end{align} 
We can similarly prove the case of $\mathrm{RR}10_{\beta}$.
The inductive step is again the same as for Proposition~\ref{prop:soundness_CR+}. \qed
\end{proof}

\subsection{Proof-Theoretic Strength of Ramified Realisability}

This subsection is devoted to proof-theoretic studies of $\mathsf{RR}_{< \gamma}$ and its variants.
First of all, the conservativity proof of $\mathsf{CR}$ (Proposition~\ref{prop:conservativity_CR}) is easily generalised to that of $\mathsf{RR}_{ < \gamma}$.

\begin{lemma}\label{lem:conservativity_RR}
$\mathsf{RR}_{<\gamma}$ is conservative over $\mathsf{PA}$ for every ordinal $\gamma$.
\end{lemma}

\begin{proof}
%By Lemma~\ref{lem:conservative_R-normality}, it suffices to show conservativity of $\mathsf{RR}_{<\gamma}^{-}$.
Similarly to the proof of Proposition~\ref{prop:conservativity_CR}, we define a translation $\mathcal{T}: \Lr^{< \gamma} \to \mathcal{L}$:
\begin{itemize}
\item $\mathcal{T}(s=t) \ = \ (s=t)$;
\item $\mathcal{T}(s \in \pole) \ = \ \mathcal{T}(s\F_{\beta}t) \ = \ \mathcal{T}(s\T_{\beta}t) \ = \ (0=0)$ \ for $\beta < \gamma$;
\item $\mathcal{T}$ commutes with the logical symbols.
\end{itemize}
Then, it is easy to prove that $\mathsf{RR}_{<\gamma} \vdash A$ implies $\mathsf{PA} \vdash \mathcal{T}(A)$. \qed
\end{proof}

%\begin{lemma}\label{lem:RR_with_empty}
%Let $\mathsf{RR}_{<\gamma}^{\emptyset}$ be 
%$\mathsf{RR}_{<\gamma}$ augmented with $\pole = \emptyset$. Then, $\mathsf{RR}_{<\gamma}^{\emptyset}$ can define the truth prediates of $\mathsf{RT}_{<\gamma}$.
%Therefore, every $\mathcal{L}$-theorem of $\mathsf{RT}_{<\gamma}$ is derivable in $\mathsf{RR}_{<\gamma}^{\emptyset}$.
%\end{lemma}

Next, we determine the proof-theoretic upper bound of $\mathsf{RR}_{< \gamma}^{\emptyset}$.
As with $\mathsf{CR}^{\emptyset}$, it is not so difficult to formalise the model $\langle \emptyset, < \gamma\rangle$ of $\mathsf{RR}_{< \gamma}^{\emptyset}$ in the system $\mathsf{RA}_{< \gamma}$ of ramified analysis (cf.~\cite{feferman1964systems}), which is known to be proof-theoretically equivalent to $\mathsf{RT}_{< \gamma}$.
But here, we give a more direct interpretation of $\mathsf{RR}_{< \gamma}^{\emptyset}$ into $\mathsf{RT}_{< \gamma}$, based on \cite{hayashi2025friedman}.

Similarly to \cite[Proposition~11]{hayashi2025friedman}, we define a translation $\mathcal{T}_{\emptyset} : \Lr^{< \gamma} \to \Lt^{< \gamma}$ as follows:
\begin{itemize}
\item $\mathcal{T}_{\emptyset}(s \in \pole) \ = \ \bot$;

\item $\mathcal{T}_{\emptyset}(s = t) \ = \ (s=t)$;

\item $\mathcal{T}_{\emptyset}(s \F_{\beta} t) \ = \ \T_{\beta} (\tau_{\emptyset} (s \dot{\in} \lVert t \rVert))$ for $\beta < \gamma$;

\item $\mathcal{T}_{\emptyset}(s \T_{\beta} t) \ = \ \T_{\beta} (\tau_{\emptyset} (s \dot{\in} \lvert t \rvert))$ for $\beta < \gamma$;

\item $\mathcal{T}_{\emptyset}$ commutes with the logical symbols,
\end{itemize}
where $\tau_{\emptyset}$ is a primitive recursive representation of $\mathcal{T}_{\emptyset}$.
Thus, we have $\mathsf{PA} \vdash \tau_{\emptyset} (\ulcorner A \urcorner) = \ulcorner \mathcal{T}_{\emptyset}(A) \urcorner$ for each $A \in \Lr^{< \gamma}$.

\begin{lemma}\label{lem:upper-bound_RR-empty}
For each $\Lr^{< \gamma}$-formula $A$, if $ \mathsf{RR}_{<\gamma}^{\emptyset} \vdash A$, then $\mathsf{RT}_{< \gamma} \vdash \mathcal{T}_{\emptyset}(A)$.
\end{lemma}

\begin{proof}
By induction on the derivation of $A$.
Since the other cases are proved in the same way as the proof of \cite[Proposition~11]{hayashi2025friedman}, we shall only deal with the axioms $\mathrm{RR}7_{\beta}$ to $\mathrm{RR}10_{\beta}$.
\begin{description}
\item[$\mathrm{RR}7_{\beta}$] Assume $\alpha < \beta < \gamma$. The translation of $\mathrm{RR}7_{\beta}$ is equivalent to the following:
\[
\forall \ulcorner A \urcorner \in \pred{Sent}_{\mathrm{R}}^{<\alpha}. \ \T_{\beta} \ulcorner \T_{\alpha} \ulcorner \mathcal{T}_{\emptyset} (\dot{a} \in \lVert \dot{b} \in \lVert A \rVert \rVert ) \urcorner \urcorner \leftrightarrow \T_{\alpha} \ulcorner \mathcal{T}_{\emptyset} (\dot{a} \in \lVert \dot{b} \in \lVert A \rVert \rVert ) \urcorner.
\]
If $\ulcorner A \urcorner \in \pred{Sent}_{\mathrm{R}}^{<\alpha}$, then $\ulcorner \mathcal{T}_{\emptyset} (\dot{a} \in \lVert \dot{b} \in \lVert A \rVert \rVert ) \urcorner \in \pred{Sent}_{\mathrm{T}}^{<\alpha}$, which is verifiable in $\mathsf{PA}$.
Thus, the above formula is derivable by $\mathrm{RT}5_{\beta}$.

\item[$\mathrm{RR}8_{\beta}$] Same as above.

\item[$\mathrm{RR}9_{\beta}$] Assume $\beta < \gamma$. The translation of $\mathrm{RR}9_{\beta}$ is equivalent to the following:
\[
\forall \delta < \beta ( \forall \ulcorner A \urcorner \in \pred{Sent}_{\mathrm{R}}^{<\delta}. \ \T_{\beta} \ulcorner \T_{\delta} \ulcorner \mathcal{T}_{\emptyset} (\dot{a} \in \lVert \dot{b} \in \lVert A \rVert \rVert) \urcorner \urcorner \leftrightarrow \T_{\beta} \ulcorner \mathcal{T}_{\emptyset} (\dot{a} \in \lVert \dot{b} \in \lVert A \rVert \rVert) \urcorner ),
\]
which is derived by $\mathrm{RT}6_{\beta}$.

\item[$\mathrm{RR}10_{\beta}$] Same as above. 
\qed
\end{description}
\end{proof}

In the remainder of this subsection, we determine the lower bound of $\mathsf{RR}_{< \gamma}^{+}$ and $\mathsf{RR}_{< \gamma}^{\emptyset}$.
In \cite{hayashi2025friedman}, the lower bound of $\mathsf{FSR}$ with the reflection rule is obtained by proving \emph{explicit realisability} of $\mathsf{FSR}$,
which states that for each theorem $A$ of $\mathsf{FSR}^{\emptyset}$ (= $\mathsf{FSR} + \pole = \emptyset$), there exists a term $t$ such that $\mathsf{FSR} \vdash t \in \lvert A \rvert$. 

Thus, we first show explicit realisability of $\mathsf{RR}_{< \gamma}$.

\begin{lemma}\label{lem:explicit-realisability-RR}
For each $\Lr^{< \gamma}$-formula $A$, if $ \mathsf{RR}^{\emptyset}_{<\gamma} \vdash A$, then there exists a term $t$ such that $\mathsf{RR}_{< \gamma} \vdash t \in \lvert A \rvert$.

Moreover, this fact is formalisable in $\mathsf{RR}_{\gamma}$, i.e. there exists a term $s$ such that the following holds:
\[
\mathsf{RR}_{\gamma} \vdash \forall \ulcorner A \urcorner \in \pred{Sent}_{\mathrm{R}}^{< \gamma}. \ \pred{Bew}_{\mathsf{RR}^{\emptyset}_{< \gamma}}(a, \ulcorner A \urcorner) \to (s \cdot a) \T_{\gamma}\ulcorner A \urcorner,
\]
where $\pred{Bew}_{\mathsf{RR}^{\emptyset}_{< \gamma}}(x,y)$ is a canonical provability predicate for $\mathsf{RR}^{\emptyset}_{< \gamma}$ which means that $x$ is the code of a proof of $y$ in $\mathsf{RR}^{\emptyset}_{< \gamma}$.
\end{lemma}

\begin{proof}
The proof is by induction on the derivation of $A$.
Firstly, the results for $\mathsf{CR}$ can easily be generalised to those for explicit and formal realisability. 
In particular, we will use Lemma~\ref{partial_compositionality}, Lemma~\ref{lem:continuation_constant}, and Theorem~\ref{formal_realisability_PA} without proof. 
Therefore, it suffices to consider the additional axioms of $\mathsf{RR}_{< \gamma}$. 
Moreover, realisability of the axioms $\mathrm{RR}1_{\beta}$ to $\mathrm{RR}6_{\beta}$ are shown exactly in the same way as in \cite[\S3]{hayashi2025friedman}. 
Similarly, the axiom $\pole = \emptyset$ is treated in \cite[Corollary~4]{hayashi2025friedman}.
Thus, we concentrate on the remaining axioms $\mathrm{RR}7_{\beta}$ to $\mathrm{RR}10_{\beta}$.
\begin{description}
\item[$\mathrm{RR}7_{\beta}$] Assume $\alpha < \beta < \gamma$. 
Note that $\mathrm{RR}7_{\beta}$ has the form of universally quantified conditional:
\begin{align}
\forall x \Bigl( \pred{Sent}_{\mathrm{R}}^{<\alpha} (x) \to  \bigl( a\F_{\beta} ( \pred{sub} ( \ulcorner b \F_{\alpha} x \urcorner, b, x ) ) \leftrightarrow a \in \lVert b\F_{\alpha} x \rVert \bigr) \Bigr) \label{lem:explicit-realisability-RR:proof0}
\end{align}
%$\forall x (\pred{Sent}_R^{< \alpha} (x) \to \cdots)$
Thus, if the antecedent $\pred{Sent}_{\mathrm{R}}^{< \alpha} (x)$ is false, then we have $0 \in \lVert \pred{Sent}_{\mathrm{R}}^{< \alpha} (x) \rVert$, and hence
the conditional of \ref{lem:explicit-realisability-RR:proof0} is realisable by some term $r$ by Lemma~\ref{lem:continuation_constant}. 
If $\pred{Sent}_{\mathrm{R}}^{< \alpha} (x)$ is true, then we can express the succedent as $a\F_{\beta} \ulcorner \dot{b} \F_{\alpha} \ulcorner A \urcorner \urcorner \leftrightarrow a \in \lVert b\F_{\alpha} \ulcorner A \urcorner \rVert$,
where $\ulcorner A \urcorner = x \in \pred{Sent}_{\mathrm{R}}^{< \alpha}$. 
To give a realiser for the succedent,
we first prove 
$\forall x (x \in \lvert a\F_{\beta} \ulcorner \dot{b} \F_{\alpha} \ulcorner A \urcorner \urcorner \rvert \leftrightarrow x \in \lvert a \in \lVert b\F_{\alpha} \ulcorner A \urcorner \rVert \rvert)$ in $\mathsf{RR}_{< \gamma}$ as follows:
\begin{align}
x \in \lvert a\F_{\beta} \ulcorner \dot{b} \F_{\alpha} \ulcorner A \urcorner \urcorner \rvert \ &\leftrightarrow \ x \T_{\beta} \ulcorner \dot{a} \in \lVert \dot{b} \F_{\alpha} \ulcorner A \urcorner \rVert \urcorner &&\text{by Corollary~\ref{cor:explicit_realisability}} \notag \\
&\leftrightarrow \ x \T_{\beta} \ulcorner \dot{a} \F_{\alpha} \ulcorner \dot{\dot{b}} \in \lVert  A \rVert \urcorner \urcorner &&\text{by Definition~\ref{defn:emplicit_realisability}} \notag \\
%&\Leftrightarrow \ x \T_{\beta} \ulcorner \dot{a} \in \lVert \dot{b} \in \lVert  A \rVert \rVert \urcorner \tag{$\mathrm{RR}9_{\beta}$ and $\mathrm{RR}2_{\beta}$} \\
&\leftrightarrow \ x \in \lvert  a \F_{\alpha} \ulcorner \dot{b} \in \lVert  A \rVert \urcorner \rvert %\tag{Lemma~\ref{lem:explicit_formal_RR}}
&&\text{by Corollary~\ref{cor:RR}} \notag \\
&\leftrightarrow x \in \lvert a \in \lVert b\F_{\alpha} \ulcorner A \urcorner \rVert \rvert &&\text{by Definition~\ref{defn:emplicit_realisability}} \notag
\end{align}
Then, it is easy to see the following:
\begin{align}
(\lambda y. \langle (y)_0, (y)_1 \rangle) \in \lvert a\F_{\beta} \ulcorner \dot{b} \F_{\alpha} \ulcorner A \urcorner \urcorner \to a \in \lVert b\F_{\alpha} \ulcorner A \urcorner \rVert \rvert. \label{lem:explicit-realisability-RR:proof1}
\end{align}
Indeed, $\mathsf{RR}_{< \gamma}$ derives:
\begin{align}
(\ref{lem:explicit-realisability-RR:proof1}) &\Leftrightarrow \forall z (z \in \lVert a\F_{\beta} \ulcorner \dot{b} \F_{\alpha} \ulcorner A \urcorner \urcorner \to a \in \lVert b\F_{\alpha} \ulcorner A \urcorner \rVert \rVert \to \langle (\lambda y. \langle (y)_0, (y)_1 \rangle), z \rangle \in \pole) \notag \\
&\Leftrightarrow \forall z \bigl( (z)_0 \in \lvert a\F_{\beta} \ulcorner \dot{b} \F_{\alpha} \ulcorner A \urcorner \urcorner \rvert \land (z)_1 \in \lVert a \in \lVert b\F_{\alpha} \ulcorner A \urcorner \rVert \rVert \to \langle (\lambda y. \langle (y)_0, (y)_1 \rangle), z \rangle \in \pole \bigr) \notag \\
&\Leftarrow \forall z \bigl( (z)_0 \in \lvert a\F_{\beta} \ulcorner \dot{b} \F_{\alpha} \ulcorner A \urcorner \urcorner \rvert \land (z)_1 \in \lVert a \in \lVert b\F_{\alpha} \ulcorner A \urcorner \rVert \rVert \to \langle (z)_0, (z)_1 \rangle \in \pole \bigr) \label{lem:explicit-realisability-RR:proof2}
\end{align}
Since we have seen $\forall z \bigl( (z)_0 \in \lvert a\F_{\beta} \ulcorner \dot{b} \F_{\alpha} \ulcorner A \urcorner \urcorner \rvert \leftrightarrow (z)_0 \in \lvert a \in \lVert b\F_{\alpha} \ulcorner A \urcorner \rVert \rvert \bigr)$, the formula (\ref{lem:explicit-realisability-RR:proof2}) and hence (\ref{lem:explicit-realisability-RR:proof1}) hold.
Similarly, we get the converse direction:
\begin{align}
(\lambda y. \langle (y)_0, (y)_1 \rangle) \in \lvert a\F_{\beta} \ulcorner \dot{b} \F_{\alpha} \ulcorner A \urcorner \urcorner \leftarrow a \in \lVert b\F_{\alpha} \ulcorner A \urcorner \rVert \rvert. \label{lem:explicit-realisability-RR:proof3}
\end{align}
As realisability is closed under classical logic (cf.~Theorem~\ref{formal_realisability_PA}), from (\ref{lem:explicit-realisability-RR:proof1}) and (\ref{lem:explicit-realisability-RR:proof3}) follows that
$a\F_{\beta} \ulcorner \dot{b} \F_{\alpha} \ulcorner A \urcorner \urcorner \leftrightarrow a \in \lVert b\F_{\alpha} \ulcorner A \urcorner \rVert$ is also realised.
Thus, the conditional of (\ref{lem:explicit-realisability-RR:proof0}) is also realised by some term $q$ by Lemma~\ref{partial_compositionality} and Theorem~\ref{formal_realisability_PA}.

Using the above terms $r,q$, we now define a term $p$ such that:
\begin{center}
$p \cdot x \simeq
\begin{cases}
q & \text{if $\pred{Sent}_{\mathrm{R}}^{< \alpha} (x)$} \\
r & \text{otherwise.}
\end{cases}$
\end{center}
Then, $p \cdot x$ always realises the conditional of (\ref{lem:explicit-realisability-RR:proof0}), whether $\pred{Sent}_{\mathrm{R}}^{< \alpha} (x)$ is true or not.  
Thus,  (\ref{lem:explicit-realisability-RR:proof0}) itself is realisable by Lemma~\ref{partial_compositionality}.

Finally, the above proof is straightforwardly formalisable by using $(\mathrm{RR}9_{\gamma})'$ of Corollary~\ref{cor:RR}.
For instance, we have in $\mathsf{RR}_{\gamma}$:
\[
\forall \beta < \gamma \bigl( \forall \alpha < \beta (
\forall \ulcorner A \urcorner \in \pred{Sent}_{\mathrm{R}}^{<\alpha}. \ x \T_{\gamma} \ulcorner \dot{a} \F_{\beta} \ulcorner \dot{\dot{b}} \F_{\alpha} \ulcorner A \urcorner \urcorner \urcorner \leftrightarrow x \T_{\gamma} \ulcorner \dot{a} \in \lVert \dot{b}\F_{\alpha} \ulcorner A \urcorner \rVert \urcorner ) \bigr),
\]
from which it follows that $\mathsf{RR}_{\gamma} \vdash \forall \beta < \gamma \bigl( \forall \alpha < \beta (t \T_{\gamma} \ulcorner (\mathrm{RR}7_{\dot{\beta}}) \urcorner) \bigr)$ for some term $t$.
 
\item[$\mathrm{RR}8_{\beta}$] Same as above.

\item[$\mathrm{RR}9_{\beta}$] By the same reason as above, we take any $\delta < \beta$ and $\ulcorner A \urcorner \in \pred{Sent}_{\mathrm{R}}^{< \alpha}$, and it is enough to give a realiser for the formula: $a\F_{\beta} \ulcorner \dot{b} \F_{\delta} \ulcorner A \urcorner \urcorner \leftrightarrow a\F_{\beta} \ulcorner \dot{b} \in \lVert A \rVert \urcorner$.
We prove that $x \in \lvert a\F_{\beta} \ulcorner \dot{b} \F_{\delta} \ulcorner A \urcorner \urcorner \rvert$ and $x \in \lvert a\F_{\beta} \ulcorner \dot{b} \in \lVert A \rVert \urcorner \rvert$ are equivalent over $\mathsf{RR}_{< \gamma}$.
\begin{align}
x \in \lvert a\F_{\beta} \ulcorner \dot{b} \F_{\delta} \ulcorner A \urcorner \urcorner \rvert \ &\Leftrightarrow \ x \T_{\beta} \ulcorner \dot{a} \in \lVert \dot{b} \F_{\delta} \ulcorner A \urcorner \rVert \urcorner &&\text{by Corollary~\ref{cor:explicit_realisability}}  \notag \\
&\Leftrightarrow \ x \T_{\beta} \ulcorner \dot{a}\F_{\delta} \ulcorner \dot{\dot{b}} \in \lVert A \rVert \urcorner \urcorner &&\text{by Definition~\ref{defn:emplicit_realisability}} \notag \\
&\Leftrightarrow \ x \T_{\beta} \ulcorner \dot{a} \in \lVert \dot{b} \in \lVert A \rVert \rVert \urcorner &&\text{by Corollary~\ref{cor:RR}} \notag \\
&\Leftrightarrow \ x \in \lvert a\F_{\beta} \ulcorner \dot{b} \in \lVert A \rVert \urcorner \rvert &&\text{by Corollary~\ref{cor:explicit_realisability}} \notag
\end{align}

\item[$\mathrm{RR}10_{\beta}$] Same as above.\qed

%\item[$\mathrm{RR}11_{\beta}$] 
%Assuming $x \in \lvert a \T_{\beta} b \rvert$, it is sufficient to uniformly find a realiser for $ \pred{Sent}^{< \beta}_R (b)$. 
%By Definition~\ref{defn:emplicit_realisability}, the assumption is equivalent to $x \T_{\beta}(a \dot{\in} \lvert b \rvert)$, which implies $\pred{Sent_R^{< \beta}}(a \dot{\in} \lvert b \rvert)$ by $\mathrm{RR}11_{\beta}$.
%Hence, we also have $\pred{Sent_R^{< \beta}}(b)$ by the definition of $\pred{Sent_R^{< \beta}}$, and thus $\lambda y.y$ is, for instance, a realiser for $\pred{Sent_R^{< \beta}}(b)$. 
\end{description}
%\qed
\end{proof}

By the previous lemma, it immediately follows that $\mathsf{RR}_{< \gamma}^+$ and $\mathsf{RR}_{< \gamma}^{\emptyset}$ are proof-theoretically equivalent.
So, to determine the lower bound of them, we prove that $\mathsf{RR}_{< \gamma}^{\emptyset}$ can relatively interpret $\mathsf{RT}_{< \gamma}$. 
For this purpose, generalising the translation in Lemma~\ref{lem:CR_with_empty},
we define a translation $\mathcal{T}_{0} : \Lt^{< \gamma} \to \Lr^{< \gamma}$ such that:
\begin{itemize}
\item $\mathcal{T}_{0}(\T_{\beta} t) \ = \  \pred{Sent}_{\mathrm{T}}^{< \beta}(t) \to 0 \T_{\beta} (\tau_0 (t)) $ for $\beta < \gamma$, 
\item the other vocabularies are unchanged.
\end{itemize}
Here, $\tau_0$ is a primitive recursive representation of $\mathcal{T}_{0}$.

The following two lemmata are provable in the same way as for Lemma~\ref{lem:CR_with_empty} (see also \cite[Proposition~5]{hayashi2025friedman}).

\begin{lemma}\label{lem:compositional-truth_RR-empty}
Take any $\beta < \gamma$. Then, the following are derivable in $\mathsf{RR}_{< \gamma}^{\emptyset}$:
\begin{enumerate}
\item $ \forall \ulcorner A \urcorner \in \pred{Sent}_{\mathrm{R}}^{< \beta}. \ x \T_{\beta} \ulcorner A \urcorner \leftrightarrow 0 \T_{\beta} \ulcorner A \urcorner$

\item $0 \T_{\beta} \ulcorner P(\dot{\vec{x}}) \urcorner  \leftrightarrow P(\vec{x}) $, for each atomic predicate $P(\vec{x})$ of $\mathcal{L}_{\pole}$

\item $ \forall \ulcorner A \urcorner, \ulcorner B \urcorner \in \pred{Sent}_{\mathrm{R}}^{< \beta}. \ 0 \T_{\beta} \ulcorner A \to B \urcorner \leftrightarrow \bigl( 0 \T_{\beta} \ulcorner A \urcorner \to 0 \T_{\beta} \ulcorner B \urcorner \bigr)$

\item $ \forall \ulcorner A_x \urcorner \in \pred{Sent}_{\mathrm{R}}^{< \beta}. \ 0 \T_{\beta} \ulcorner \forall x A \urcorner \leftrightarrow \forall x (0 \T_{\beta} \ulcorner A(\dot{x}) \urcorner)$
\end{enumerate}
\end{lemma}

\begin{lemma}\label{lem:compositional-truth_PA-empty}
For any $\mathcal{L}_{\mathrm{R}}^{< \gamma}$-formulas $A,B$, 
the following are derivable in $\mathsf{RR}_{< \gamma}^{\emptyset}$:
\begin{enumerate}
\item $x \in \lvert A \rvert \leftrightarrow 0 \in \lvert A \rvert$

\item $0 \in \lvert P(\vec{x}) \rvert \leftrightarrow P(\vec{x})$, for each atomic predicate $P(\vec{x})$ of $\mathcal{L}_{\pole}$

\item $0 \in \lvert A \to B \rvert \leftrightarrow (0 \in \lvert A \rvert \to 0 \in \lvert B \rvert)$

\item $0 \in \lvert \forall x A \rvert \leftrightarrow \forall x (0 \in \lvert A(x) \rvert)$
\end{enumerate}
\end{lemma}

The next lemma states that explicit realisability of $\mathcal{T}_0 (A)$ is, in the presence of $\pole = \emptyset$, equivalent to $\mathcal{T}_0 (A)$ itself.

\begin{lemma}\label{lem:T-sentence_RR} 
Let $A$ be an $\Lt^{< \gamma}$-formula.
Then, the following holds:
\[
\mathsf{RR}_{< \gamma}^{\emptyset} \vdash  0 \in \lvert \mathcal{T}_0 (A) \rvert \leftrightarrow \mathcal{T}_0 (A).
\]
Moreover, this is formalisable in $\mathsf{PA}$ in the sense that there exists a code $e$ of a partial recursive function such that
\[
\mathsf{PA} \vdash \forall \ulcorner A \urcorner \in \pred{Sent}_{\mathrm{T}}^{< \gamma}. \ \pred{Bew}_{\mathsf{RR}_{< \gamma}^{\emptyset}} (\num e \cdot \langle \gamma, \ulcorner A \urcorner \rangle, \ulcorner 0 \in \lvert \mathcal{T}_0 (A) \rvert \leftrightarrow \mathcal{T}_0 (A) \urcorner).
\]
\end{lemma}

\begin{proof}
We prove the claim by induction on $\gamma$ and on the logical complexity of $A$.
If $A$ is an atomic predicate of $\mathcal{L}_{\pole}$, then the claim is proved in a similar way as for Lemma~\ref{lem:CR_with_empty}.
Thus, we consider the case $A = \T_{\beta}t$ for some $\beta < \gamma$.
First, we have $\mathcal{T}_0 (\T_{\beta}t) \leftrightarrow \bigl( \pred{Sent}_{\mathrm{T}}^{< \beta}(t) \to 0 \T_{\beta} (\tau_0(t)) \bigr)$ and hence the following:
\begin{align}
0 \in \lvert \mathcal{T}_0 (\T_{\beta}t) \rvert \ &\Leftrightarrow \ 0 \in \lvert \pred{Sent}_{\mathrm{T}}^{< \beta}(t) \to 0 \T_{\beta} (\tau_0(t)) \rvert &&\text{by Definition of $\mathcal{T}_0$} \notag \\
%&\Leftrightarrow \ 0 \in \lvert \pred{Sent}_T^{< \beta}(t) \to 0 \T_{\beta} (\tau_0(t)) \rvert \tag{Lemma~\ref{lem:compositional-truth_PA-empty}} \\
%&\Leftrightarrow \ 0 \in \lvert \pred{Sent}_T^{< \beta}(t) \rvert \to 0 \in \lvert 0 \T_{\beta} (\tau_0(t)) \rvert \tag{Lemma~\ref{lem:compositional-truth_PA-empty}} \\
&\Leftrightarrow \ \pred{Sent}_{\mathrm{T}}^{< \beta}(t) \to 0 \in \lvert 0 \T_{\beta} (\tau_0(t)) \rvert &&\text{by Lemma~\ref{lem:compositional-truth_PA-empty}} \notag \\
&\Leftrightarrow \ \pred{Sent}_{\mathrm{T}}^{< \beta}(t) \to 0 \T_{\beta} (0 \dot{\in} \lvert \tau_0(t) \rvert ) &&\text{by Corollary~\ref{cor:explicit_realisability}} \notag
\end{align}
Thus, if $t$ does not denote the code of an $\Lt^{< \beta}$-sentence, then both $\mathcal{T}_0 (\T_{\beta}t)$ and $0 \in \lvert \mathcal{T}_0 (\T_{\beta}t) \rvert$ are true, and thus the claim follows.
Otherwise, we can write $A = \T_{\beta}\ulcorner B \urcorner$ for some $\ulcorner B \urcorner = t \in \pred{Sent}_{\mathrm{T}}^{< \beta}$, and we now need to prove $0 \T_{\beta} (\tau_0(\ulcorner B \urcorner)) \leftrightarrow 0 \T_{\beta} (0 \dot{\in} \lvert \tau_0(\ulcorner B \urcorner) \rvert )$, i.e. $0 \T_{\beta} \ulcorner \mathcal{T}_0 (B) \urcorner \leftrightarrow 0 \T_{\beta} \ulcorner 0 \in \lvert \mathcal{T}_0(B) \rvert \urcorner$.
But, by the main induction hypothesis,
it follows that $\mathsf{PA} \vdash \forall \ulcorner B \urcorner \in \pred{Sent}_{\mathrm{T}}^{< \beta}. \ \pred{Bew}_{\mathsf{RR}_{< \beta}^{\emptyset}} (\num e \cdot \langle \beta, \ulcorner B \urcorner \rangle, \ulcorner 0 \in \lvert \mathcal{T}_0 (B) \rvert \leftrightarrow \mathcal{T}_0 (B) \urcorner)$.
%Note here that the formalisation requires only mathematical induction for $\mathcal{L}$, and hence $\mathsf{PA}$ is sufficient.
Therefore, by Lemma~\ref{lem:explicit-realisability-RR}, 
we particularly have some term $s$ such that $\mathsf{RR}_{\beta} \vdash s \T_{\beta} \ulcorner 0 \in \lvert \mathcal{T}_0 (B) \rvert \leftrightarrow \mathcal{T}_0 (B) \urcorner$.
Thus, Lemma~\ref{lem:compositional-truth_RR-empty} implies that $\mathsf{RR}_{< \gamma}^{\emptyset} \vdash 0 \T_{\beta}\ulcorner 0 \in \lvert \mathcal{T}_0 (B) \rvert \urcorner \leftrightarrow 0 \T_{\beta}\ulcorner \mathcal{T}_0 (B) \urcorner$,
as required.
%\begin{align}
%x \in \lvert \mathcal{T}_0 (\T_{\beta} \ulcorner B \urcorner) \rvert \ &\Leftrightarrow \ x \in \lvert 0 \T_{\beta} \ulcorner \mathcal{T}_0 (B) \urcorner \rvert \tag{Definition of $\mathcal{T}_0$} \\
%&\Leftrightarrow  \ x \T_{\beta} \ulcorner 0 \in \lvert \mathcal{T}_0 (B) \rvert \urcorner \tag{Definition~\ref{defn:emplicit_realisability}} \\
%&\Leftrightarrow  \ 0 \T_{\beta} \ulcorner 0 \in \lvert \mathcal{T}_0 (B) \rvert \urcorner \tag{$\pole = \emptyset$} \\
%&\Leftrightarrow \ 0 \T_{\beta} \ulcorner \mathcal{T}_0 (B) \rvert \urcorner \notag \\
%&\Leftrightarrow \ \mathcal{T}_0 (\T_{\beta} \ulcorner B \urcorner) \tag{Definition of $\mathcal{T}_0$}
%\end{align}

If $A$ is a complex formula, then the proof is again along the lines of that of Lemma~\ref{lem:CR_with_empty}.

Finally, by formalising the above, we can find an appropriate partial recursive function $\{ e \}$ such that $\mathsf{PA}$ can derive \emph{reflexive progressiveness} of the claim $C(\gamma) := \forall \ulcorner A \urcorner \in \pred{Sent}_{\mathrm{T}}^{< \gamma}. \ \pred{Bew}_{\mathsf{RR}_{< \gamma}^{\emptyset}} (\num e \cdot \langle \gamma, \ulcorner A \urcorner \rangle, \ulcorner 0 \in \lvert \mathcal{T}_0 (A) \rvert \leftrightarrow \mathcal{T}_0 (A) \urcorner)$, i.e. 
\[
\mathsf{PA} \vdash \forall \gamma \bigl( \forall \beta < \gamma (\pred{Bew}_{\mathsf{PA}} \ulcorner C(\dot{\beta}) \urcorner) \to C(\gamma) \bigr),
\]
where $\pred{Bew}_{\mathsf{PA}}(x)$ is a canonical provability predicate for $\mathsf{PA}$.
Therefore, by Schmerl's trick (\cite[p.~337]{schmerl1979fine}), we obtain $C(\gamma)$ itself in $\mathsf{PA}$.
\qed
\end{proof}

Now we get relative interpretability of $\mathsf{RT}_{< \gamma}$, as desired.

\begin{lemma}\label{lem:interpretation_RT}
Let $A$ be an $\Lt^{< \gamma}$-formula.
If $\mathsf{RT}_{< \gamma} \vdash A$, then $\mathsf{RR}_{< \gamma}^{\emptyset} \vdash \mathcal{T}_0 (A)$.
\end{lemma}

\begin{proof}
By induction of the derivation of $A$.
As the other cases are immediate from Corollary~\ref{cor:RR} and Lemma~\ref{lem:compositional-truth_RR-empty}
(see also Lemma~\ref{lem:CR_with_empty} or  \cite[Proposition~5]{hayashi2025friedman}),
we only consider the axioms $\mathrm{RT}5_{\beta}$ and $\mathrm{RT}6_{\beta}$.
\begin{description}
\item[$\mathrm{RT}5_{\beta}$] 
$\mathcal{T}_0 (\forall \ulcorner A \urcorner \in \pred{Sent}_{\mathrm{T}}^{<\alpha}. \ \T_{\beta} \ulcorner \T_{\alpha} \ulcorner A \urcorner \urcorner \leftrightarrow \T_{\alpha} \ulcorner A \urcorner)$ is equivalent to the formula:
\[
\forall \ulcorner A \urcorner \in \pred{Sent}_{\mathrm{T}}^{<\alpha}. \  \mathcal{T}_0 (\mathcal{T}_0 (\T_{\beta} \ulcorner \T_{\alpha} \ulcorner A \urcorner \urcorner) \leftrightarrow \mathcal{T}_0 (\T_{\alpha} \ulcorner A \urcorner).
\]
Taking any $\ulcorner A \urcorner \in \pred{Sent}_{\mathrm{T}}^{< \alpha}$, we then prove equivalence of $\mathcal{T}_0(\T_{\beta} \ulcorner \T_{\alpha} \ulcorner A \urcorner \urcorner)$ and $ \mathcal{T}_0 (\T_{\alpha} \ulcorner A \urcorner)$ as follows:
\begin{align}
& \mathcal{T}_0(\T_{\beta} \ulcorner \T_{\alpha} \ulcorner A \urcorner \urcorner) \notag \\
&\Leftrightarrow \pred{Sent}_{\mathrm{T}}^{< \beta} (\ulcorner \T_{\alpha} \ulcorner A \urcorner \urcorner) \to 0 \T_{\beta} \ulcorner \mathcal{T}_0 (\T_{\alpha} \ulcorner A \urcorner ) \urcorner &&\text{by Definition of $\mathcal{T}_0$} \notag \\
&\Leftrightarrow 0\T_{\beta} \ulcorner \mathcal{T}_0 (\T_{\alpha} \ulcorner A \urcorner) \urcorner &&\text{by $\pred{Sent}_{\mathrm{T}}^{< \beta} (\ulcorner T_{\alpha} \ulcorner A \urcorner \urcorner)$} \notag \\
&\Leftrightarrow 0 \T_{\beta} \ulcorner \pred{Sent}_{\mathrm{T}}^{< \alpha} (\ulcorner A \urcorner) \to 0 \T_{\alpha} \ulcorner \mathcal{T}_0 (A) \urcorner \urcorner &&\text{by Definition of $\mathcal{T}_0$} \notag \\
&\Leftrightarrow \pred{Sent}_{\mathrm{T}}^{< \alpha} (\ulcorner A \urcorner) \to 0\T_{\beta} \ulcorner 0 \T_{\alpha} \ulcorner \mathcal{T}_0 (A) \urcorner \urcorner &&\text{by Lemma~\ref{lem:compositional-truth_RR-empty}} \notag \\
&\Leftrightarrow \pred{Sent}_{\mathrm{T}}^{< \alpha} (\ulcorner A \urcorner) \to 0\in \lvert 0 \T_{\alpha} \ulcorner \mathcal{T}_0 (A) \urcorner \rvert &&\text{by Corollary~\ref{cor:RR}} \notag \\
&\Leftrightarrow 0\in \lvert \pred{Sent}_{\mathrm{T}}^{< \alpha} (\ulcorner A \urcorner) \to 0 \T_{\alpha} \ulcorner \mathcal{T}_0 (A) \urcorner \rvert &&\text{by Lemma~\ref{lem:compositional-truth_PA-empty}} \notag \\
&\Leftrightarrow 0 \in \lvert \mathcal{T}_0 (\T_{\alpha} \ulcorner A \urcorner) \rvert &&\text{by Definition of $\mathcal{T}_0$} \notag \\
&\Leftrightarrow \mathcal{T}_0 (\T_{\alpha} \ulcorner A \urcorner) &&\text{by Lemma~\ref{lem:T-sentence_RR}} \notag
\end{align}

\item[$\mathrm{RT}6_{\beta}$] 
$\mathcal{T}_0 \bigl( \forall \delta ( \delta < \beta \to \forall \ulcorner A \urcorner \in \pred{Sent}_{\mathrm{T}}^{<\delta}. \ \T_{\beta} \ulcorner \T_{\delta} \ulcorner A \urcorner \urcorner \leftrightarrow \T_{\beta} \ulcorner A \urcorner ) \bigr)$ is equivalent to the formula:
\[
\forall \delta \bigl( \delta < \beta \to \forall \ulcorner A \urcorner \in \pred{Sent}_{\mathrm{T}}^{<\delta}. \ \mathcal{T}_0  (\T_{\beta} \ulcorner \T_{\delta} \ulcorner A \urcorner \urcorner) \leftrightarrow \mathcal{T}_0 (\T_{\beta} \ulcorner A \urcorner ) \bigr).
\]
Thus, taking any $\delta < \beta$ and $\ulcorner A \urcorner \in \pred{Sent}_{\mathrm{T}}^{< \delta}$, we want to prove the equivalence of $\mathcal{T}_0 (\T_{\beta} \ulcorner \T_{\delta} \ulcorner A \urcorner \urcorner)$ and $\mathcal{T}_0 (\ulcorner \T_{\beta} \ulcorner A \urcorner \urcorner)$.
Firstly, we have the following equivalences in the same way as above:
\begin{align}
\mathcal{T}_0 (\T_{\beta} \ulcorner \T_{\delta} \ulcorner A \urcorner \urcorner) &\Leftrightarrow 
\pred{Sent}_{\mathrm{T}}^{< \beta}(\ulcorner \T_{\delta} \ulcorner A \urcorner \urcorner) \to 
0 \T_{\beta} \ulcorner \mathcal{T}_0 ( \T_{\delta} \ulcorner A \urcorner ) \urcorner \notag \\
&\Leftrightarrow 
0 \T_{\beta} \ulcorner \mathcal{T}_0 ( \T_{\delta} \ulcorner A \urcorner ) \urcorner \notag \\
&\Leftrightarrow 0 \T_{\beta} \ulcorner \pred{Sent}_{\mathrm{T}}^{< \delta}(\ulcorner A \urcorner) \to 0 \T_{\delta} \ulcorner \mathcal{T}_0 (A)\urcorner \urcorner
\notag \\
&\Leftrightarrow \pred{Sent}_{\mathrm{T}}^{< \delta}(\ulcorner A \urcorner) \to  0 \T_{\beta} \ulcorner 0 \T_{\delta} \ulcorner \mathcal{T}_0 (A)\urcorner \urcorner
\notag \\
&\Leftrightarrow 0 \T_{\beta} \ulcorner 0 \T_{\delta} \ulcorner \mathcal{T}_0 (A)\urcorner \urcorner \notag \\
%%%%%%%%%%%%%%%%%%%%%%%%%%%%%%%%%%%%%
\mathcal{T}_0 (\ulcorner \T_{\beta} \ulcorner A \urcorner \urcorner) &\Leftrightarrow \pred{Sent}_{\mathrm{T}}^{< \beta} (\ulcorner A \urcorner) \to 0 \T_{\beta} \ulcorner \mathcal{T}_0 (A) \urcorner \notag \\
&\Leftrightarrow 0 \T_{\beta} \ulcorner \mathcal{T}_0 (A) \urcorner \notag
\end{align}
We therefore prove $0 \T_{\beta} \ulcorner 0 \T_{\delta} \ulcorner \mathcal{T}_0 (A)\urcorner \urcorner \leftrightarrow 0 \T_{\beta} \ulcorner \mathcal{T}_0 (A) \urcorner$ in the following.

By Lemma~\ref{lem:T-sentence_RR}, we have $\mathsf{PA} \vdash \pred{Bew}_{\mathsf{RR}_{< \beta}^{\emptyset}}( s, \ulcorner 0 \in \lvert \mathcal{T}_0 (A)\rvert \leftrightarrow \mathcal{T}_0 (A) \urcorner )$ for some term $s$.
%for every $\delta < \beta$ and every $A \in \pred{Sent}_T^{< \delta}$.
By Lemma~\ref{lem:explicit-realisability-RR}, this is formally realisable by some term $t$ in $\mathsf{RR}_{\beta}$:
\begin{align}
%\forall \delta < \beta \bigl( \forall \ulcorner A \urcorner \in \pred{Sent}_T^{<\delta}. \ 
\mathsf{RR}_{\beta} \vdash t \T_{\beta} \ulcorner 0 \in \lvert \mathcal{T}_0 (A)\rvert \leftrightarrow \mathcal{T}_0 (A) \urcorner, \notag
%\label{proof-RT6:lem:interpretation_RT}
\end{align}
%where $t$ is the code of some partial recursive function that returns a realiser of $0 \T_{\delta} \ulcorner \mathcal{T}_0 (A)\urcorner \leftrightarrow \mathcal{T}_0 (A)$ when applied to the pair $\langle \delta, \ulcorner A \urcorner \rangle$.

which, by Lemma~\ref{lem:compositional-truth_RR-empty}, %(\ref{proof-RT6:lem:interpretation_RT}) is 
implies:
\begin{align}
%\forall \delta < \beta \bigl( \forall \ulcorner A \urcorner \in \pred{Sent}_T^{<\delta}. \ 
\mathsf{RR}_{\beta}^{\emptyset} \vdash 0 \T_{\beta} \ulcorner 0 \in \lvert \mathcal{T}_0 (A)\rvert \urcorner \leftrightarrow 0 \T_{\beta} \ulcorner \mathcal{T}_0 (A) \urcorner 
%\bigr). 
\label{proof-RT6.2:lem:interpretation_RT}
\end{align}
Since $\ulcorner A \urcorner \in \pred{Sent}_{\mathrm{T}}^{< \delta}$ implies $\ulcorner \mathcal{T}_0 (A) \urcorner \in \pred{Sent}_{\mathrm{R}}^{< \delta}$, the formula (\ref{proof-RT6.2:lem:interpretation_RT}) is, by Corollary~\ref{cor:RR}, equivalent to 
$0 \T_{\beta} \ulcorner 0 \T_{\delta} \ulcorner \mathcal{T}_0 (A)\urcorner \urcorner \leftrightarrow 0 \T_{\beta} \ulcorner \mathcal{T}_0 (A) \urcorner$, as required. \qed

%Thus, the formula (\ref{proof-RT6.3:lem:interpretation_RT}) is equivalent to $\mathcal{T}_0 (\mathrm{RT}6_{\beta})$. \qed
\end{description}
\end{proof}

\begin{theorem}\label{thm:realisability-strength_RR}
All of $\mathsf{RR}^{+}_{< \gamma}$, $\mathsf{RR}^{\emptyset}_{< \gamma}$, and $\mathsf{RT}_{< \gamma}$ have the same $\mathcal{L}$-theorems for any ordinal $\gamma$.
In particular, if $\gamma = \omega^{\lambda} \geq \omega$, then they are also $\mathcal{L}$-equivalent to $\mathsf{PA} + \mathsf{TI}({<} \varphi (1 + \lambda) 0)$.
\end{theorem}

\begin{proof}
Samely as the proof of Lemma~\ref{lem:reflection_empty_CR}, we can verify that $\mathsf{RR}^{+}_{< \gamma}$ is a subtheory of $\mathsf{RR}^{\emptyset}_{< \gamma}$.
On the other hand, $\mathsf{RR}^{\emptyset}_{< \gamma}$ is, by Lemma~\ref{lem:explicit-realisability-RR}, realisable in $\mathsf{RR}_{< \gamma}$
Thus, $\mathsf{RR}^{+}_{< \gamma}$ derives every $\mathcal{L}$-theorem of $\mathsf{RR}^{\emptyset}_{< \gamma}$ by the reflection rule.

Next, $\mathsf{RR}^{\emptyset}_{< \gamma}$ is relatively interpretable in $\mathsf{RT}_{< \gamma}$ by Lemma~\ref{lem:upper-bound_RR-empty}.
Conversely, $\mathsf{RT}_{< \gamma}$
is relatively interpretable in $\mathsf{RR}_{< \gamma}^{\emptyset}$ by Lemma~\ref{lem:interpretation_RT}.

Finally, it is well known that $\mathsf{RT}_{< \gamma}$ is conservative over $\mathsf{PA} + \mathsf{TI}({<} \varphi (1 + \lambda) 0)$ when $\gamma = \omega^{\lambda} \geq \omega$ (see e.g. \cite[Theorem~4.4]{leigh2016reflecting}). \qed
\end{proof}

%%%%%%%%%%%%%%%%%%%%%%%%%%%%%%%%%%%%%%%%%%%%%%%%%%%

\section{Future Work}\label{sec:classical_real_future_work}
In this paper, we have axiomatised Krivine's classical realisability in a similar manner to the formalisation of Tarskian hierarchical truth, and then we generalise it to ramified theories.
Given that various self-referential approaches to truth have been developed \cite{cantini1990theory,friedman1987axiomatic,kripke1976outline,leitgeb2005truth},
it is natural to consider self-referential generalisations of classical realisability.
Since a Friedman--Sheard-style system is already proposed in \cite{hayashi2025friedman}, the next step would be to formulate systems based on Kripkean theories of truth.

Another direction of future work is the formalisation of alternative interpretations for classical theories.
Alternative realisability interpretations for $\mathsf{PA}$ and its extensions are presented in, e.g., \cite{aschieri2010learning,aschieri2012interactive,avigad2000realizability,berger2005modified,blot2015typed,blot2017realizability}.
It is also reasonable to consider the axiomatisation of intuitionistic realisability interpretations over Heyting arithmetic.  

%While this paper is mainly concerned with the proof-theoretic study of classical realisability,
%there seems room for further model-theoretic study. In particular, if the classical realisability can be seen as a generalisation of the concept of classical truth, perhaps the model-theoretic results for classical truth can be generalised to those for classical realisability.

%Secondly, we can think of self-referential systems for classical realisability, similar to various self-referential theories of truth. 
%For sxample, the theory $\mathsf{FS}$ consists of the generalised compositional axioms in which 

%As for philosophical implications, we fistly remark that the central concept in classical realisability is falsity or refutation rather than truth or realisation. Secondly, the concept of truth seems to be reduced to realisability or verifiability.
%Therefore, our results might be a rebuttal of some kind of primitivism on truth. \cite{naibo2016verificationism} may be seen as such an attempt, where the possibility of verificationist theory of meaning for classical logic is argued.

%%%%%%%%%%%%%%%%%%%%%%%%%%%%%%%%%%%%%%%%%%%%%

%\begin{credits}
%\subsubsection{\ackname}
%
%\end{credits}

 \bibliographystyle{splncs04}
 \bibliography{inyou}
%Author, F.: Article title. Journal \textbf{2}(5), 99--110 (2016)

%Author, F., Author, S.: Title of a proceedings paper. In: Editor,
%F., Editor, S. (eds.) CONFERENCE 2016, LNCS, vol. 9999, pp. 1--13.
%Springer, Heidelberg (2016). \doi{10.10007/1234567890}

%Author, F., Author, S., Author, T.: Book title. 2nd edn. Publisher,

%Author, A.-B.: Contribution title. In: 9th International Proceedings
%on Proceedings, pp. 1--2. Publisher, Location (2010)

%LNCS Homepage, \url{http://www.springer.com/lncs}, last accessed 2023/10/25
%\end{thebibliography}

%%%%%%%%%%%%%%%%%%%%%%%%%%%%%%%

\end{document}